\documentclass[11pt, a4paper]{article}

\RequirePackage{amsthm,amsmath,amsfonts,amssymb}
\RequirePackage[authoryear]{natbib} 
\RequirePackage[colorlinks,citecolor=blue,urlcolor=blue,breaklinks]{hyperref}

\usepackage{bm}
\usepackage{mathrsfs}  
\usepackage{appendix}
\usepackage{commath}
\usepackage{dsfont}
\usepackage{subfig}
\usepackage{xcolor}
\usepackage{mathtools}
\usepackage{enumitem}
\usepackage{fancyhdr}
\usepackage{geometry}
\usepackage{caption}
\usepackage{booktabs}
\usepackage{multirow}
\usepackage{float}
\usepackage{algorithm}
\usepackage{algpseudocode}

\usepackage{xcolor}
\usepackage{setspace}
\newcommand{\N}{\mathbb{N}}
\newcommand{\R}{\mathbb{R}}
\newcommand{\eps}{\varepsilon}
\renewcommand{\P}{\mathbb{P}}
\newcommand\E{\mathbb{E}}
\newcommand{\mc}{\mathcal}

\theoremstyle{plain}
\newtheorem{theorem}{Theorem}[section]
\newtheorem{lemma}[theorem]{Lemma}

\newtheorem{proposition}[theorem]{Proposition}

\theoremstyle{remark}
\newtheorem{remark}{Remark}[section]

\newtheorem{assumption}{Assumption}[section]

\begin{document}

\title{Adversarially robust multiple testing in high dimensions
}

\author{
\begin{tabular}{c}
Anders Bredahl Kock\footnote{Kock's research was supported by the European Research Council (ERC) grant number 101124535 -- HIDI (UKRI EP/Z002222/1).  He is also a member of, and grateful for support from, i) the Aarhus Center for Econometrics (ACE), funded by the Danish National Research Foundation grant number DNRF186,  and ii) the Center for Research in Energy: Economics and Markets (CoRE).} \\ 
\footnotesize	University of Oxford \\
\footnotesize Department of Economics\\
\footnotesize	10 Manor Rd, Oxford OX1 3UQ
\\
\footnotesize	{\footnotesize	\href{mailto:anders.kock@economics.ox.ac.uk}{anders.kock@economics.ox.ac.uk}} 
\end{tabular}
\begin{tabular}{c}
David Preinerstorfer \\ 
{\footnotesize WU Vienna University of Economics and Business} \\
{\footnotesize Institute for Statistics and Mathematics} \\
{\footnotesize Welthandelsplatz 1, 1020 Vienna} \\ 
{\footnotesize	 \href{mailto:david.preinerstorfer@wu.ac.at}{david.preinerstorfer@wu.ac.at}}
\end{tabular}
}

\date{Preliminary version: July, 2026}

\maketitle	

\begin{abstract}
Robust multiple testing procedures for assessing equality restrictions on the coordinates of high-dimensional mean vectors are proposed. Our procedures are based on quantile-winsorization techniques, approximately control the familywise error rate (strongly) under adversarial contamination, and allow the number of hypotheses to grow exponentially with sample size, despite requiring only slightly more than two moments. Technically, our one-sample results build on recent Gaussian approximation inequalities for the distribution of high-dimensional quantile-winsorized means, whereas our two-sample results are based on extensions thereof, which we develop here and which could be of some interest in their own right. 
\end{abstract}

\section{Introduction}

Developing methods robust to adversarial contamination has attracted considerable interest in the modern statistics literature, cf., e.g.,~\cite{lai2016agnostic}, \cite{cheng2019high}, \cite{diakonikolas2019robust}, \cite{hopkins2020robust}, \cite{LM21}, \cite{minsker2021robust}, \cite{bhatt2022minimax}, \cite{depersin2022robust}, \cite{dalalyan2022all}, \cite{minasyan2023statistically}, \cite{minsker2023efficient}, \cite{oliveira2025finite}, and \cite{Wins1}. The observational assumption in this setting is that the statistician cannot access the original sample, but only observes a contaminated version of it, without knowing \emph{which} values have been contaminated or \emph{how} they have been contaminated. In addition to the contaminated sample, a non-random upper bound on the proportion of contaminated observations is available to the statistician. When conducting inference on the mean, as observed in \cite{LM21}, adversarial robustness, i.e., good statistical properties despite contamination, can be achieved through appropriately trimming or winsorizing the observations, cf.~also \cite{Wins1}. 

In addition to adversarial robustness, recent results in \cite{resende2024robust}, \cite{liu2024robust}, and~\cite{kphdapprox} uncovered another benefit of working with robust, instead of arithmetic, means in high-dimensional settings: Besides their robustness to contamination, suitably trimmed/winsorized means possess advantages in terms of the dimension-dependence in Gaussian approximation results that are uniformly valid without requiring strong moment or tail decay conditions.

Given the two just-mentioned benefits, from a statistical point of view, one may immediately wonder whether inference based on robust means necessarily entails some loss in efficiency. \cite{kprobfree} show that tests for a (high-dimensional) global null hypothesis based on quantile-winsorized means can be constructed that: (i)~are robust to adversarial contamination, (ii)~enjoy favorable dimension dependence despite relatively weak moment assumptions, and (iii)~do not imply a loss in power at first asymptotic order. They consider the one-sample case and focus on testing a global zero restriction on the unknown population mean, contributing to a large literature on tests for equality restrictions on high-dimensional means for the one-sample or multi-sample case, cf., e.g., the articles~\cite{hotelling1931}, \cite{dempster1958high}, \cite{bai1996effect}, \cite{tony2014two}, \cite{wang2015high}, \cite{xu2016adaptive}, \cite{xue2020distribution}, \cite{zhang2020simple}, \cite{kock2023consistency, kock2024enhanced}, \cite{yang2024new}, \cite{jiang2024nonparametric}, \cite{qiu2025self}, and the recent overview in~\cite{huang2022overview}. 

In many practical situations, however, it is not of terminal interest whether the global null hypothesis can be rejected, the setting which was studied in \cite{kprobfree}. Rather, the statistician intends to answer \emph{which} of the null hypotheses are in fact violated. Treating the latter question properly requires a \emph{multiple} testing procedure. Multiple testing procedures based on (non-robust) high-dimensional arithmetic means and corresponding Gaussian approximation results were investigated in~\cite{chernozhukov2013gaussian}, cf.~also~\cite{belloni2018high}, their approach crucially relying on general purpose multiple testing algorithms in \cite{romano2005exact}.  Arithmetic mean-based procedures naturally break down in the adversarially contamination setting, as such means (as well as the corresponding covariances) are genuinely non-robust, cf., e.g.,~\cite{kprobfree} for illustrations in a hypothesis testing framework. Furthermore, to allow for the dimension to increase exponentially in sample size in the Gaussian approximation results in~\cite{chernozhukov2013gaussian}, it is necessary to impose relatively strong moment conditions, cf.~also \cite{zhangwu2017} and~\cite{kock2024remark}.

In the present paper, we extend the multiple testing approach taken by~\cite{chernozhukov2013gaussian} and~\cite{belloni2018high} (i.e., combining high-dimensional Gaussian approximations with insights from \cite{romano2005exact}) to the adversarially contaminated setting. We base our multiple testing procedures on quantile-winsorization methods for location and covariance estimation. 

In the \emph{one-sample case}, we can apply the Gaussian approximation results for high-dimensional winsorized means recently obtained in \cite{kphdapprox} and~\cite{kprobfree}. This allows us to obtain adversarially robust multiple testing procedures with performance guarantees on their familywise error rate in finite samples in a relatively straightforward way. To cover also the \emph{two-sample case}, for which Gaussian approximations are currently unavailable in the literature, we here develop suitable high-dimensional Gaussian approximation results. This includes the development of a covariance estimator that satisfies suitable concentration properties, and which we develop in the pooled and non-pooled case. 

In both, the one-sample and two-sample case, we develop procedures based on ``worst-case'' critical values (Algorithms~\ref{algo:gauss} and~\ref{algo:gaussts}), exact (but numerically infeasible) critical values from a Gaussian with estimated covariance matrix (Algorithms~\ref{algo:boot} and~\ref{algo:bootts}), and approximate (bootstrap) critical values from a Gaussian with estimated covariance matrix (Algorithms~\ref{algo:boot2} and~\ref{algo:bootts2}). Taking the inaccuracy in determining the critical values into account in the analysis appears to be new in this context. Our theoretical main contributions (Theorems~\ref{thm:mainthG},~\ref{thm:mainthB},~\ref{thm:mainthB2},~\ref{thm:mainthGts},~\ref{thm:mainthBts},~\ref{thm:mainthBts2}) provide finite sample upper bounds on the familywise error rate. The bounds are valid uniformly over large families of distributions, requiring only mild moment conditions (cf.~Assumptions~\ref{ass:setting} and~\ref{ass:settingY}). We discuss under which asymptotic scenarios said upper bounds converge to the pre-scribed level (cf.~Assumptions~\ref{ass:asyX} and~\ref{ass:asyXY}).

Besides being useful for constructing multiple testing procedures, the two-sample results may be of some interest in their own right. For example, we construct and briefly verify (asymptotic) size control of three tests for the global null hypothesis in the two-sample case; cf.~Theorem~\ref{thm:mainthGtsgl}. Our tests can be viewed as robust versions of the test for the global hypothesis studied in \cite{xue2020distribution}, the analysis of which required extending one-sample approximation results from \cite{chernozhukov2017central} to the two-sample setting. In contrast to their test, which is based on non-robust means, we investigate a normalized test statistic based on quantile-winsorized estimators, which is adversarially robust. Their procedure imposes stronger moment conditions (we only require slightly more than~$2$ moments) and is not robust to adversarial contamination. 

The results for the one-sample case are discussed in Section~\ref{sec:onesample}. Two-sample testing problems are covered in Section~\ref{sec:ts}. All proofs are collected in the Appendices~\ref{app:aux}-\ref{app:proofsts}.

\section{One-sample multiple testing under contamination}\label{sec:onesample}

\subsection{Observational setting: the one-sample case}\label{sec:osobs}

To set the stage,\footnote{Several quantities introduced in the one-sample setting carry the index~``$X$.'' This can be considered somewhat superfluous in the one-sample setting, but turns out convenient later when considering the two-sample setting in Section~\ref{sec:ts}. } let~$X_1, \hdots, X_{n_X}$ be~$d$-dimensional i.i.d.~random vectors with expectation~$\mu^X := \E(X_1)$ and covariance matrix~$\Sigma^X := \E((X_1 - \mu^X)(X_1 - \mu^X)')$, say, and consider the multiple (hypothesis) testing problem
\begin{equation}\label{eq:testingproblem}
\mathsf{H}_{0,j}: \mu^X_j = 0 \qquad\text{vs.}\qquad \mathsf{H}_{1,j}: \mu^X_j \neq 0, \qquad \text{for } j = 1, \hdots, d.
\end{equation}
\emph{Throughout this article, we assume~$d \geq 2$, an assumption that entails no loss of generality in a multiple testing setting, and~$n_X > 3$.}

As pointed out in the introduction, we are interested in a scenario where the statistician does \emph{not} have direct access to~$X_1, \hdots, X_{n_X}$. Rather, an ``adversary'' first inspects~$X_1, \hdots, X_{n_X}$, and then hands over a \emph{contaminated} sample of random vectors $\tilde{X}_{1},\hdots,\tilde{X}_{n_X}$ in~$\R^d$ to the statistician satisfying
\begin{equation}\label{eq:contamfrac}
\left|\cbr[1]{i\in\cbr[0]{1,\hdots,n_X}:\tilde{X}_{i}\neq X_{i}}\right|
\leq
\overline{\eta}^X n_X.
\end{equation}
The quantity~$\overline{\eta}^X  \in[0,1/2)$ denotes a \emph{known and non-random} upper bound on the proportion of contaminated observations. It is not required that~$\overline{\eta}^X$ be the smallest such upper bound (although smaller bounds reduce winsorization in the methods we introduce below and hence lead to a more efficient use of the data). For easy reference, we summarize the setting laid out above, together with some technical (moment) conditions, in the following assumption. We denote~$\sigma^X_{m,j} := [\E(|X_{1,j}- \mu^X_j|^m)]^{1/m}$ for~$m \geq 1$.
\begin{assumption}\label{ass:setting}
The~$X_1,\hdots,X_{n_X}$ are i.i.d.~random vectors in~$\R^d$ with~$\E |X_{1,j}|^{m_X}<\infty$ for some~$m_X\in(2,\infty)$ and all~$j=1,\hdots,d$,~$\mu^X := \E(X_1)$, and~$\Sigma^X := \E((X_1 - \mu^X)(X_1 - \mu^X)')$. There exist strictly positive constants~$b^X_1$ and~$b^X_2$ such that~$\min_{j=1,\hdots,d}\sigma^X_{2,j}\geq b^X_1$ and $\sigma^X_{m_X}:=\max_{j=1,\hdots,d}\sigma^X_{m_X,j}\leq b^X_2$. The actually observed adversarially contaminated random vectors (in~$\R^d$) are denoted~$\tilde{X}_1,\hdots,\tilde{X}_{n_X}$ and satisfy~\eqref{eq:contamfrac}.	
\end{assumption}

Most of our results below are finite sample results. Hence, we may think of, e.g.,~$\overline{\eta}^X$ depending on~$n_X$, or~$d$ depending on~$n_X$, without showing this explicitly in our notation. To lighten notation, we will abstain from showing the dependence of several quantities on sample size or dimension whenever this causes no confusion.

\subsection{Winsorized means, covariance matrix estimators, self-normalized statistics}

For real numbers~$x_1,\hdots,x_l$, denote by~$x_1^*\leq \hdots\leq x_l^*$ their non-decreasing rearrangement, and for any pair of \emph{winsorization points}~$-\infty<\alpha\leq\beta<\infty$, define the \emph{winsorization function}
\begin{equation*}
\phi_{\alpha,\beta}(x)
:=
\begin{cases}
	\alpha\qquad \text{if }x<\alpha,\\
	x\qquad \text{if }x\in[\alpha,\beta],\\
	\beta\qquad \text{if }x>\beta.
\end{cases}
\end{equation*} 
Given~$\lambda_{1, X}$ and~$\lambda_{1, X}'$ in~$(1, \infty)$ as well as~$\lambda_{2, X}$ and~$\lambda_{2, X}'$ in~$(0, \infty)$, we introduce the \emph{winsorization parameters} for mean and covariance estimation, respectively,
\begin{equation}\label{eq:epsfamx}
\eps_{X} =:\lambda_{1, X}\cdot \overline{\eta}^X +\lambda_{2, X}\cdot \frac{\log(dn_X)}{n_X} \quad \text{ and } \quad \eps_{X}':=\lambda_{1, X}' \cdot  \overline{\eta}^X + \lambda_{2, X}'\cdot \frac{\log(d^2n_X)}{n_X}.
\end{equation}
We emphasize that the symbols~$\eps_{X}$ and~$\eps_{X}'$ (as well as their two-sample analogues that will be introduced in Section~\ref{sec:ts}) \emph{throughout always} refer to the quantities in~\eqref{eq:epsfamx}, without mentioning this explicitly. 
In the subsequent definitions (see also \cite{kphdapprox}, where the estimators were introduced with somewhat different notation, cf.~also Footnote~\ref{rem:notKP}), we shall tacitly assume that the winsorization parameters are smaller than~$1/2$ (possible violations of that assumption, implying that some of the subsequent quantities are not well defined, will be taken care of separately in the definition of our multiple testing procedures). Let the lower and upper winsorization points for the~$j$-th coordinate be defined as~
\begin{equation}\label{eqn:winsp}
\hat{\alpha}^X_j:=\tilde X_{\lceil \eps_X n_X \rceil,j}^* \quad \text{ and } \quad \hat{\beta}^X_j:=\tilde X^*_{\lfloor(1-\eps_X)n_X\rfloor + 1,j},
\end{equation}
and define~$S_{W}^{\dagger, X}\in\R^d$, the properly scaled winsorized location estimator, via its coordinates
\begin{equation}\label{eqn:winsmeandef}
S^{\dagger, X}_{W, j}
:=
n_X^{-1/2}\sum_{i=1}^{n_X}\phi_{\hat\alpha_j^X,\hat\beta_j^X}(\tilde{X}_{i,j}), \quad j=1,\hdots,d.
\end{equation}
Likewise, we define its centered version~$S_{W, S}^{\dagger, X}\in\R^d$ via
\begin{equation}\label{eqn:winsmeandefc}
S^{\dagger, X}_{W, S, j}
:=
n_X^{-1/2}\sum_{i=1}^{n_X} \left(\phi_{\hat\alpha_j^X,\hat\beta_j^X}(\tilde{X}_{i,j}) - \mu^X_j \right), \quad j=1,\hdots,d.
\end{equation}
To estimate~$\Sigma^X$, we use a similar winsorization idea: for~$j=1,\hdots,d$ set~
\begin{equation}\label{eqn:winsp2}
\hat{a}^X_j:=\tilde{X}_{\lceil \eps_X' n_X \rceil,j}^*, \quad \hat{b}^X_j:=\tilde{X}_{\lfloor (1-\eps_X')n_X \rfloor + 1,j}^*, \quad \text{ and } \quad  \tilde{\mu}^X_{j}:=n_X^{-1}\sum_{i=1}^{n_X}\phi_{\hat{a}^X_j,\hat{b}^X_j}(\tilde{X}_{i,j}),
\end{equation}
and define the (positive semi-definite and symmetric, cf.~Remark~\ref{rem:gramrep} below) covariance estimator~$\tilde{\Sigma}^X$ via
\begin{equation}\label{eq:tildeSigma}
\tilde{\Sigma}^X_{j,k}
:=
n_X^{-1}\sum_{i=1}^{n_X}\sbr[1]{\phi_{\hat{a}^X_j,\hat{b}^X_j}(\tilde{X}_{i,j})-\tilde{\mu}^X_{j}}\sbr[1]{\phi_{\hat{a}^X_k,\hat{b}^X_k}(\tilde{X}_{i,k})-\tilde{\mu}^X_{k}},\qquad 1\leq j,k\leq d.
\end{equation}

\begin{remark}[Choosing~$\lambda_{1, X}, \lambda_{1, X}', \lambda_{2, X}, \lambda_{2, X}'$ in practice]\label{rem:tpch}
We make the following remarks concerning the practical choice of the tuning parameters underlying the winsorization parameters in~\eqref{eq:epsfamx}. 
\begin{enumerate}
\item The parameters~$\lambda_{1, X} > 1$ and~$\lambda_{1, X}' > 1$ guarantee that at least~$\overline{\eta}^X n_X$ observations are winsorized on the lower and upper range of the contaminated data. It is clear that in the present adversarial contamination setting, cf.~in particular Equation~\eqref{eq:contamfrac}, this amount of winsorization cannot be avoided. Based on numerical results in~\cite{kphdapprox}, we suggest to choose~$\lambda_{1, X}$ and~$\lambda_{1, X}'$ close to~$1$, e.g., equal to~$1.01$. [Note that in settings where there is no adversarial contamination, i.e., where~$\overline{\eta}^X = 0$ which we do not rule out,~$\lambda_{1, X}$ and~$\lambda_{1, X}'$ play no role.]
\item The parameters~$\lambda_{2, X} > 0$ and~$\lambda_{2, X}' > 0$ are technical artifacts from the application of concentration inequalities in the proofs underlying \cite{Wins1}, cf.~also \cite{LM21}, and should be chosen close to zero. In practice, again following numerical observations in~\cite{kphdapprox}, we suggest~$\lambda_{2, X} = 0.1$ and ~$\lambda_{2, X}' = 0.07$, and note that for this choice, the number of observations winsorized in our estimators for~$\mu^X$ and~$\Sigma^X$, respectively, is about the same.
\end{enumerate}
\end{remark}

Conducting statistical inference based on~$S_{W}^{\dagger, X}$ directly, could imply disadvantages in settings where the variances of the coordinates of~$X_i$ differ, a case which occurs frequently in practice. Therefore, we base our multiple testing procedures on a normalized version: let~$\tilde{\sigma}^X_{j}$,~$j=1,\hdots,d$, denote the square roots of the diagonal elements of~$\tilde{\Sigma}^X$, and define the normalized analogue~$S_{W}^X$ to~$S_{W}^{\dagger, X}$ via 
\begin{equation}\label{eqn:SXWdef}
S_{W,j}^{X}
=
\frac{S^{\dagger, X}_{W, j}}{\tilde{\sigma}^X_{j}}
=
\frac{1}{\sqrt{n_X}\tilde{\sigma}^X_{j}}\sum_{i=1}^{n_X}\phi_{\hat\alpha^X_j,\hat\beta^X_j}(\tilde{X}_{i,j}),\quad j=1,\hdots,d;
\end{equation}
for completeness, we set~$S_{W,j}^{X} = 0$ in case~$\tilde{\sigma}^X_{j} = 0$. Furthermore, we denote the correlation matrix corresponding to~$\tilde{\Sigma}^X$ by~$\tilde{\Sigma}^X_{0}$; i.e., for any pair of indices~$i,j$ we define 
\begin{equation}\label{eqn:corrmatdef}
\tilde{\Sigma}^X_{0, i, j} := 
\begin{cases}
\tilde{\Sigma}^X_{i,j}/\sqrt{\tilde{\Sigma}^X_{i,i}\tilde{\Sigma}^X_{j,j}} & \text{ if } \tilde{\Sigma}^X_{i,i} \times \tilde{\Sigma}^X_{j,j} > 0, \\
0 & \text{ if } \tilde{\Sigma}^X_{i,i} \times \tilde{\Sigma}^X_{j,j} = 0 \text{ and } i \neq j, \\
1 & \text{ if } \tilde{\Sigma}^X_{i,i}= 0 \text{ and } i = j.
\end{cases}
\end{equation}
Denoting~$\tilde{D}^X := \mathrm{diag}(\tilde{\sigma}^X_{1}, \hdots, \tilde{\sigma}^X_{d})$, we immediately see that~$\tilde{\Sigma}^X_{0} = (\tilde{D}^X)^{-1} \tilde{\Sigma}^X (\tilde{D}^X)^{-1}$ in case all diagonal elements of~$\tilde{D}^X$ are positive. Furthermore,~$\tilde{\Sigma}^X_{0}$ is a correlation matrix also in those cases where a variance estimator vanishes (and where~$(\tilde{D}^X)^{-1} \tilde{\Sigma}^X (\tilde{D}^X)^{-1}$ would hence not be a well-defined expression).

\begin{remark}\label{rem:gramrep}
Denoting~$$\tilde{\mu}^X = (\tilde{\mu}^X_1, \hdots, \tilde{\mu}^X_d)' \quad \text{ and } \quad \phi_{\bm{\hat{a}}^X, \bm{\hat{b}}^X}(\tilde{X}_i) := (\phi_{\hat{a}^X_1,\hat{b}^X_1}(\tilde{X}_{i,1}), \hdots, \phi_{\hat{a}^X_d,\hat{b}^X_d}(\tilde{X}_{i,d}))',$$ it clearly holds that
\begin{equation}\label{eqn:gramrep}
\tilde{\Sigma}^X = \tilde{G}^{X} [\tilde{G}^{X}]' \quad \text{ for } \quad \tilde{G}^{X} := n_X^{-1/2} \left[\phi_{\bm{\hat{a}}^X, \bm{\hat{b}}^X}(\tilde{X}_1) - \tilde{\mu}^X, \hdots, \phi_{\bm{\hat{a}}^X, \bm{\hat{b}}^X}(\tilde{X}_{n_X}) - \tilde{\mu}^X\right];
\end{equation}
[in other words,~$\tilde{\Sigma}^X$ is just the sample covariance matrix (the scaling being by sample size) of the quantile-winsorized contaminated observations].
Correspondingly, if the variances in~$\tilde{D}^X$ are all non-zero, we have
\begin{equation}\label{eqn:gramrep0}
\tilde{\Sigma}^X_0=  [\tilde{D}^X]^{-1}\tilde{G}^{X} [\tilde{G}^{X}]'[\tilde{D}^X]^{-1}.
\end{equation}
The representations in~\eqref{eqn:gramrep} and~\eqref{eqn:gramrep0} can be computationally convenient when one needs to sample from, e.g.,~$\mathsf{N}_d(0, \tilde{\Sigma}^X)$ or~$\mathsf{N}_d(0, \tilde{\Sigma}^X_0)$, which can then be achieved by pre-multiplying by~$\tilde{G}^{X}$ or by~$[\tilde{D}^X]^{-1}\tilde{G}^{X}$, respectively, a sample from~$\mathsf{N}_{n_X}(0, I_{n_X})$, which is typically referred to as a ``Gaussian multiplier bootstrap'' in the literature mentioned in the introduction.
\end{remark}

Before we move on, let use observe that in case one of the winsorization parameters is close to~$1/2$, one winsorizes the bottom and top~50\% of the observations in each coordinate, which cannot work well (in particular when it comes to variance estimation). Although our procedures are well defined for~$\varepsilon_X$ or~$\varepsilon_X'$ (smaller but) close to~$1/2$, and although our theoretical guarantees are finite sample results, they are mostly informative in situations where~$\varepsilon_X$ or~$\varepsilon_X'$ are small. Therefore, in addition to finite sample upper bounds, all our results also contain asymptotic statements, illustrating when the upper bounds can be expected to be close to the targeted value. [An analogous observation also applies to the results in Section~\ref{sec:ts} and will not be restated.] 
Specifically, we formulate asymptotic convergence statements in the following regime:
\begin{assumption}\label{ass:asyX}
	It holds that~$n_X \to \infty$, and that~$d = d(n_X)$ and~$\overline{\eta}_X = \overline{\eta}_X(n_X)$ satisfy
	\begin{equation}\label{eqn:convc} \sqrt{n_X\log(d)}\overline{\eta}_X^{1-\frac{1}{m_X}}\to 0 \qquad \text{ and } \qquad \log(d)/n_X^{\frac{m_X-2}{5m_X-2}}\to 0.
	\end{equation}
\end{assumption}

\subsection{Large-scale one-sample multiple testing}\label{sec:osmt}

Based on uniform Gaussian approximation statements over hyper-rectangles for the winsorized means~$S^X_{W}$ developed in \cite{kphdapprox}, which we restate in a way convenient to our endeavor in Theorems~\ref{thm:HDGauss} and~\ref{thm:HDGauss_studentized} in Appendix~\ref{app:aux},\footnote{Note that \cite{kphdapprox} show the dependence of several quantities on sample size, for example, whereas we have chosen not to do this at several places for the sake of readability.\label{rem:notKP}} and the ideas to construct generic step-down methods for multiple testing problems in~\cite{romano2005exact}, we shall now propose robust step-down procedures~$\hat{\Pi} = \hat{\Pi}(\tilde{X}_1, \hdots, \tilde{X}_{n_X}) \subseteq \{1, \hdots, d\}$, say, to identify the true null hypotheses among~\eqref{eq:testingproblem}, i.e.,~$$J(\mu^X) := \left\{j \in [d] : \mu^X_j = 0\right\} \quad \text{ where} \quad [d]:=\cbr[0]{1,\hdots, d}.$$ Our main goal here is to obtain procedures that (under Assumption~\ref{ass:setting}) approximately control the familywise error rate (FWER), i.e., the probability~$\P(J(\mu^X) \not \subseteq \hat{\Pi} )$ that at least one true null hypothesis is not contained in~$\hat{\Pi}$. Note that we aim to control the FWER in the ``strong'' sense, i.e., we do not assume that~$J(\mu^X) = [d]$, which would only allow us to establish ``weak'' control of the FWER, cf., e.g., \cite{dudoit2003multiple}.

We note that properties of robust and non-robust tests for the global (as opposed to the multiple) testing problem corresponding to~\eqref{eq:testingproblem} have been studied in \cite{kprobfree}.

\subsubsection{Gaussian and bootstrap-based quantile functions and uniform approximations}\label{sec:quants}

In the following discussion, we let~$Z \sim N_d(0, I_d)$ be independent of~$\tilde{X}_1,\hdots, \tilde{X}_n$. To define the (worst-case) critical values for our Gaussian-approximation based multiple testing procedure, we proceed as follows: For any non-empty set~$A \subseteq [d]$, we denote the quantile function\footnote{\label{foot:qf}Given a cumulative distribution function (cdf)~$F$ on~$\R$, we denote by~$Q_F(\gamma) := \inf \{x \in \R: F(x) \geq \gamma\}$, for~$\gamma \in (0, 1)$, its quantile function.} of (the cdf of)~$\max_{j = 1, \hdots, |A|}|Z_j|$ by~$[0, 1] \ni \beta \mapsto c_{\beta, A}$, with the convention that~$c_{0, A} := 0$ and~$c_{1, A} := \infty$. Obviously,
\begin{equation}\label{eq:MonCVG}
c_{\cdot,A_1}\leq c_{\cdot,A_2} \quad \text{ whenever } \quad \emptyset \neq A_1\subseteq A_2\subseteq[d],
\end{equation}
which extends to analogous conditional monotonicity statements in case the sets~$A_1$ or~$A_2$ depend on the data. Note that the just-introduced critical values are based on a Gaussian distribution with independent coordinates (which corresponds to the worst-case correlation structure in our case; cf.~Equation~\eqref{eqn:ksiquant} below). It is well known, and easy to verify, that the following closed-form expression holds ($\Phi^{-1}$ denoting the quantile function of the standard normal distribution)
\begin{equation}\label{eqn:Gqcf}
c_{\beta, A} = \Phi^{-1} \left[
(\beta^{1/|A|} + 1)/2
\right].
\end{equation}
To define the critical values our (multiplier) bootstrap-based multiple testing procedures exploit, we first define, for a non-empty set~$\hat{A} := \hat{A}(\tilde{X}_1,\hdots, \tilde{X}_{n_X}) \subseteq [d]$, the quantile function corresponding to the distribution of~$\max_{j\in \hat{A}}|(\tilde{\Sigma}^{X, 1/2}_{0}Z)_j|$ conditionally on~$\tilde{X}_1,\hdots, \tilde{X}_{n_X}$ by~$[0, 1] \ni \beta \mapsto c_{\beta,\hat{A}}(\tilde{\Sigma}^{X}_{0})$, i.e.,~$c_{0,\hat{A}}(\tilde{\Sigma}^{X}_{0}) := 0$,~$c_{1,\hat{A}}(\tilde{\Sigma}^{X}_{0}) := \infty$, and\footnote{To ensure that equality can always be achieved in~\eqref{eqn:bootqdef} by a suitable choice of a quantile, we used that the cdf of~$\max_{j\in \hat{A}}|(\tilde{\Sigma}^{X, 1/2}_{0}Z)_j|$ conditionally on~$\tilde{X}_1,\hdots, \tilde{X}_{n_X}$ is continuous, because its distribution possesses an everywhere strictly positive Lebesgue-density on~$(0, \infty)$ (cf.~also~\eqref{eqn:corrmatdef}). Note also that we use~$\tilde{\Sigma}^{X, 1/2}_{0}$ as a shorthand for~$(\tilde{\Sigma}^{X}_{0})^{1/2}$; and will use similar notation in later sections without further discussing this explicitly, e.g., in Section~\ref{sec:bootqts}.\label{foot:quantinv2}}
\begin{equation}\label{eqn:bootqdef}
\P\del[2]{\max_{j\in \hat{A}}|(\tilde{\Sigma}^{X, 1/2}_{0}Z)_j|\leq c_{\beta,\hat{A}}(\tilde{\Sigma}^{X}_{0})\mid \tilde{X}_1,\hdots, \tilde{X}_{n_X}}=\beta.
\end{equation} 
Note that for~$\hat{A}_1\subseteq \hat{A}_2\subseteq[d]$, it holds that~$\max_{j\in \hat{A}_1}|(\tilde{\Sigma}^{X, 1/2}_{0}Z)_j| \leq \max_{j\in \hat{A}_2}|(\tilde{\Sigma}^{X, 1/2}_{0}Z)_j|$, translating into the inequality relation on the (conditional) quantile functions
\begin{equation}\label{eq:MonCVB}
c_{\cdot,\hat{A}_1}(\tilde{\Sigma}^{X}_{0}) \leq c_{\cdot,\hat{A}_2}(\tilde{\Sigma}^{X}_{0}) \quad \text{ whenever } \quad \emptyset \neq \hat{A}_1\subseteq \hat{A}_2\subseteq[d].
\end{equation}

For controlling the FWER of bootstrap-based multiple testing procedures, Proposition~\ref{prop:BSAlwaysvalid} in Appendix~\ref{app:proofsos} relates the just-defined (conditional) quantile function to the (conditional) quantile function~$[0, 1] \ni \beta \mapsto c_{\cdot, \hat{A}}(\Sigma_0^X)$, say, of the random variable~$\max_{j\in \hat{A}}|(\Sigma^{X, 1/2}_{0}Z)_j|$ (again with the convention that~$c_{0, \hat{A}}(\Sigma^X_0) := 0$ and~$c_{1, \hat{A}}(\Sigma^X_0) := \infty$). Here, the estimator~$\tilde{\Sigma}_0^X$ was replaced by the true target~$\Sigma_0^X$, and that again (conditionally)
\begin{equation}\label{eq:MonCVBapprox}
c_{\cdot, \hat{A}_1}(\Sigma_0^X) \leq c_{\cdot, \hat{A}_2}(\Sigma_0^X) \quad \text{ whenever } \quad \emptyset \neq \hat{A}_1\subseteq \hat{A}_2\subseteq[d].
\end{equation}

We also note that the Khatri-{\v{S}}id{\'a}k inequality (e.g., from Corollary 2.4.6 in~\cite{gine2016mathematical}; cf.~\cite{khatri1967certain} and \cite{vsidak1967rectangular}) implies, for every~$\hat{A}$ as above, that (conditionally)
\begin{equation}\label{eqn:ksiquant}
c_{\cdot, \hat{A}}(\Sigma_0^X) \leq c_{\cdot, \hat{A}} \quad \text{ and } \quad c_{\cdot, \hat{A}}(\tilde{\Sigma}_0^X) \leq c_{\cdot, \hat{A}}.
\end{equation}

To formulate our finite sample bounds efficiently, we make use of the abbreviations
\begin{equation}\label{eqn:abcdef}
\begin{aligned}
\mathfrak{A}^X_{n_X} &:= \sbr[3]{\frac{\log^{5-\frac{2}{m_X}}(dn_X)}{n_X^{1-\frac{2}{m_X}}}}^{\frac{1}{4}}
+ 
\sbr[3]{\overline{\eta}_X^{1-\frac{1}{m_X}}+\sbr[2]{\frac{\log(dn_X)}{n_X}}^{1-\frac{1}{m_X}}}\sqrt{n_X\log(d)}  \\ & \hspace{2cm} + \log(d) \left[ \overline{\eta}_X^{1-\frac{2}{m_X}}+\left[\frac{\log(dn_X)}{n_X}\right]^{1-\frac{2}{m_X}} \right]^{1/2},\\[4pt]
\mathfrak{B}^X_{n_X} &:= \sqrt{\log(d)\log(dn_X)} \times \mathfrak{d}^X_{n_X} \quad \text{ for } \quad \mathfrak{d}^X_{n_X} := \overline{\eta}_X^{1-\frac{2}{m_X}}+\del[2]{\frac{\log(dn_X)}{n_X}}^{1-\frac{1}{(m_X/2)\wedge 2}}, \\
\mathfrak{C}^X_{n_X} &:= \log(d)\times \sqrt{ \mathfrak{d}^X_{n_X}}.
\end{aligned}
\end{equation}
Note that these quantities all depend on~$d, n_X, 
\overline{\eta}_X$ and~$m_X$, but that we only show their dependence on~$n_X$ explicitly for notational convenience. 
\subsubsection{Multiple testing procedures and theoretical guarantees}

Our first multiple testing procedure~$\hat{\Pi}_G$, say, is defined in Algorithm~\ref{algo:gauss}. It is built on worst-case critical values (cf.~the relation in~\eqref{eqn:ksiquant}) from a Gaussian approximation result concerning the winsorized normalized averages~$S^X_{W}$, which was obtained in~\cite{kphdapprox}, and which we collect in Theorem~\ref{thm:HDGauss_studentized} in Appendix~\ref{app:aux}. 
\begin{algorithm}
\caption{Step-down procedure for testing \eqref{eq:testingproblem} based on Gaussian critical values with worst-case correlation matrix (i.e., under independence)}
\label{algo:gauss}
\begin{algorithmic}[1]
\State \textbf{Input:} $\tilde{X}_1,\hdots, \tilde{X}_{n_X}$,~$\lambda_{1,X}, \lambda_{1,X}', \lambda_{2,X}, \lambda_{2,X}'$,~$\overline{\eta}^X$ (all as above), and~$\alpha \in (0, 1)$.
\If{$\max(\eps_X, \eps_X') \geq 1/2$}
\State $\hat{\Pi}_G \gets [d]$
\Else \State \textbf{Initialize:}~$k \gets 0$ and~$\hat{\Pi}_0 \gets[d]$

\While{$\hat{\Pi}_k \neq \emptyset$ and $\cbr[1]{j\in \hat{\Pi}_k: |S^X_{W,j}|\leq c_{1-\alpha, \hat{\Pi}_k}} \subsetneqq \hat{\Pi}_k$}  
\State $k \gets k+1$
\State $\hat{\Pi}_k \gets \cbr[1]{j\in \hat{\Pi}_{k-1}: |S^X_{W,j}|\leq c_{1-\alpha, \hat{\Pi}_{k-1}}}$
\EndWhile
\EndIf
\State \textbf{Return:} $\hat{\Pi}_G \gets \hat{\Pi}_k$.
\end{algorithmic}
\end{algorithm}

Clearly, Algorithm~\ref{algo:gauss} terminates after at most~$d$ steps (this statement applies to all our algorithms, which are step-down methods, and will not be repeated). We now establish an upper bound on the FWER of~$\hat{\Pi}_G$ (constants~$C$ may change from statement to statement, although we may use the same symbol).
\begin{theorem}\label{thm:mainthG}
Let Assumption~\ref{ass:setting} be satisfied and fix~$\alpha \in (0, 1)$. Then,
\begin{equation}\label{eqn:fwerG}
\P\left(J(\mu^X) \not \subseteq \hat{\Pi}_G  \right) \leq \alpha  +  C \times \left( \mathfrak{A}^X_{n_X} + \mathfrak{B}^X_{n_X} \right),
\end{equation}
where $C$ is a constant depending only on~$b^X_1,b^X_2,\lambda_{1, X},\lambda_{2, X},\lambda_{1, X}',\lambda_{2, X}'$ and~$m_X$. The upper bound in~\eqref{eqn:fwerG} converges to~$\alpha$ in any asymptotic regime satisfying Assumption~\ref{ass:asyX}. 
\end{theorem}

The upper bound in~\eqref{eqn:fwerG} only depends (in addition to~$\alpha$) on population quantities via the bounds in Assumption~\ref{ass:setting}, hence establishing control of the FWER uniformly over large classes of distributions satisfying only fairly mild moment bounds. Furthermore, despite only imposing only~$m_X > 2$ moments, the bound converges to~$\alpha$ in asymptotic regimes allowing for (i)~exponential growth of~$d$ in sample size~$n_X$, and (ii)~a diverging number of contaminated observations. 

The procedure defined in Algorithm~\ref{algo:gauss} is based on worst-case Gaussian critical values that do not exploit the correlation structure of~$\Sigma_0^X$. This has computational advantages, as covariances need not be computed to conduct the algorithm (note that variances have to be computed, because the test statistic for location estimation is normalized) and the quantiles can be obtained in closed form, cf.~\eqref{eqn:Gqcf}. But, at the same time, this comes at the cost of being conservative in case the corresponding observations are strongly correlated. In this case, the critical values Algorithm~\ref{algo:gauss} is based on are too large, possibly resulting in overly large sets~$\hat{\Pi}_G$. 

This conservatism can be remedied by basing the critical values on our estimator of~$\Sigma^X_0$, which results in an (idealized, cf.~the discussion after Theorem~\ref{thm:mainthB}) algorithm with similar statistical guarantees, at the cost of increased computational complexity. The method is defined in Algorithm~\ref{algo:boot}. Note that the critical values here do \emph{not} come in closed form, but have to be numerically determined (strictly speaking exactly, cf.~the discussion and motivation of Algorithm~\ref{algo:boot2} below, which avoids this subtle issue). It is common practice to approximate the critical values numerically by simulation (making use, e.g., of the Gaussian multiplier bootstrap based on the Gram-matrix representation mentioned in Remark~\eqref{rem:gramrep})
\begin{algorithm}
\caption{Step-down procedure for testing \eqref{eq:testingproblem} based on Gaussian critical values with estimated correlation matrix}
\label{algo:boot}
\begin{algorithmic}[1]
\State \textbf{Input:} $\tilde{X}_1,\hdots, \tilde{X}_{n_X}$,~$\lambda_{1,X}, \lambda_{1,X}', \lambda_{2,X}, \lambda_{2,X}'$,~$\overline{\eta}^X$ (all as above), and~$\alpha \in (0, 1)$.
\If{$\max(\eps_X, \eps_X') \geq 1/2$}
\State $\hat{\Pi}_B \gets [d]$
\Else
\State \textbf{Initialize:}~$k \gets 0$ and~$\hat{\Pi}_0 \gets[d]$
\While{$\hat{\Pi}_k \neq \emptyset$ and $\cbr[1]{j\in \hat{\Pi}_k: |S^X_{W,j}|\leq c_{1-\alpha, \hat{\Pi}_k}(\tilde{\Sigma}^{X}_{0})} \subsetneqq \hat{\Pi}_k$}  
\State $k \gets k+1$
\State $\hat{\Pi}_k \gets \cbr[1]{j\in \hat{\Pi}_{k-1}: |S^X_{W,j}|\leq c_{1-\alpha, \hat{\Pi}_{k-1}}(\tilde{\Sigma}^{X}_{0})}$
\EndWhile
\EndIf
\State \textbf{Return:} $\hat{\Pi}_B \gets \hat{\Pi}_k$.
\end{algorithmic}
\end{algorithm}

We next establish an upper bound on the FWER of $\hat{\Pi}_B$. The proof of the following result is notably more involved than that of Theorem~\ref{thm:mainthG}. The reason is that the critical values Algorithm~\ref{algo:boot} is based on are data-dependent, while the high-dimensional Gaussian approximation results we intend to use are indexed over (data-independent) hyperrectangles. Therefore, we first relate the data-dependent critical values to data-independent critical values in the course of the proof; we here make use of Proposition~\ref{prop:BSAlwaysvalid}, cf.~also the discussion before Equation~\eqref{eq:MonCVBapprox}. 
\begin{theorem}\label{thm:mainthB}
Let Assumption~\ref{ass:setting} be satisfied and fix~$\alpha \in (0, 1)$. Then,  
\begin{equation}\label{eqn:fwerB}
\P\left(J(\mu^X) \not \subseteq \hat{\Pi}_B  \right) \leq \alpha + C \times \left( \mathfrak{A}^X_{n_X} + \mathfrak{B}^X_{n_X} + \mathfrak{C}^X_{n_X}\right),
\end{equation}
where $C$ is a constant depending only on~$b^X_1,b^X_2,\lambda_{1, X},\lambda_{2, X},\lambda_{1, X}',\lambda_{2, X}'$ and~$m_X$. The upper bound in~\eqref{eqn:fwerB} converges to~$\alpha$ in any asymptotic regime satisfying Assumption~\ref{ass:asyX}.
\end{theorem}

The bound in Theorem~\ref{thm:mainthB} is similar to the one in Theorem~\ref{thm:mainthG}, but figures an additional term~$\mathfrak{C}^X_{n_X}$, stemming from an application of Proposition~\ref{prop:BSAlwaysvalid}. 

A practical disadvantage of Algorithm~\ref{algo:boot} (and also with related non-robust procedures in the literature), is that it formally requires, in each step, the \emph{exact} critical value~$c_{1-\alpha, \hat{\Pi}_k}(\tilde{\Sigma}_0^X)$. While critical values also show up in Algorithm~\ref{algo:gauss} above, the closed-form expression in~\eqref{eqn:Gqcf} can be used there (hence, ignoring numerical inaccuracies, e.g., rounding errors, this poses no problem there). The issue with Algorithm~\ref{algo:boot} is that, for general~$\tilde{\Sigma}_0^X$, there is no closed-form expression available for~$c_{1-\alpha, \hat{\Pi}_k}(\tilde{\Sigma}_0^X)$. Hence, strictly speaking, Algorithm~\ref{algo:boot} is ``idealized'' in that it is not directly implementable (or at least somewhat incomplete) as it stands. 

When implementing Algorithm~\ref{algo:boot} in practice, one has to determine the critical values numerically, e.g., by a Gaussian multiplier bootstrap based Monte Carlo method in each step~$k$. But this introduces (in the worst case~$d$) additional approximation steps, the upper bound in Theorem~\ref{thm:mainthB} is silent about: Theorem~\ref{thm:mainthB}, as is typical with bootstrap-based procedures, does not take errors in determining the critical value into account, as it works with the (idealized) Algorithm~\ref{algo:boot} based on~$c_{1-\alpha, \hat{\Pi}_k}(\tilde{\Sigma}_0^X)$. However, in a step-wise procedure, errors in determining the critical values may ``add-up'' (as opposed to a global testing problem, e.g., \cite{kprobfree}, where one only needs to determine a single critical value, and hence the computational budget for determining a single critical value is typically larger). Hence, in a multiple testing context, in particular, this subtlety deserves attention.

To provide a result for an algorithm that is feasible, without requiring exact knowledge of~$c_{1-\alpha, \hat{\Pi}_k}(\tilde{\Sigma}_0^X)$, we propose a modification of Algorithm~\ref{algo:boot}, which we state in Algorithm~\ref{algo:boot2} below. Instead of exact critical values, it incorporates explicitly the bootstrap samples and carefully chosen bootstrap (empirical) quantiles. Algorithm~\ref{algo:boot2} can therefore be viewed as one particular (feasible) implementation of Algorithm~\ref{algo:boot}. We note that the empirical quantiles it works with are \emph{not}~$(1-\alpha)$-quantiles, but \emph{require} some ``slack'' to account for the fact that the quantiles are estimated and therefore necessarily inaccurate. We briefly summarize the elementary observation underlying our construction in the following lemma (where one should think of~$B$, the number of bootstrap samples in our later applications of that lemma, as ``large''). It provides some control on the probability that a carefully chosen empirical quantile exceeds a given population quantile (recall that~$Q_F$ denotes the quantile function corresponding to a cdf~$F$ on~$\R$, cf.~Footnote~\ref{foot:qf}).
\begin{lemma}\label{lem:auxB}
Let~$z_1, \hdots, z_B$ be i.i.d.~random variables from a continuous distribution~$F$ on the real line. Denote by~$Q_{B, \beta}(\cdot)$ the quantile function of a binomial distribution with sample size~$B$ and success probability~$\beta \in (0, 1)$. Then, for every~$\beta \in (0, 1)$ and~$\gamma \in (0, 1)$ satisfying~$Q_{B, \beta}(\gamma) < B$, it holds that
\begin{equation}
\P\left(z^*_{Q_{B, \beta}(\gamma) + 1} \geq Q_F(\beta)\right) \geq \gamma.
\end{equation}
\end{lemma}
\begin{remark}
Note that~$z^*_{Q_{B, \beta}(\gamma) + 1}$ in Lemma~\ref{lem:auxB} constitutes the~$((Q_{B, \beta}(\gamma)  + 1)/B)$th empirical quantile of the observations~$z_1, \hdots, z_B$. Ignoring details concerning the specific observations~$z_i$ that we will work with for a moment, let us mention that we shall apply Lemma~\ref{lem:auxB} with~$\beta = (1-\alpha)$ and~$\gamma = (1-1/B)^{1/d}$, delivering that \emph{all} empirical quantiles used in Algorithm~\ref{algo:boot2} exceed the (targeted) population quantiles with high probability if~$B$ is large; cf.~also Remark~\ref{rem:empord}.
\end{remark}
In the formulation of Algorithm~\ref{algo:boot2}, we use the same notation (for binomial quantiles) as in Lemma~\ref{lem:auxB}, and, for a vector~$z \in \R^d$ and a non-empty~$A \subseteq [d]$, we let~$\|z\|_{\infty, A} := \max_{j \in A} |z_j|$.
\begin{algorithm}
\caption{Step-down procedure for testing \eqref{eq:testingproblem} based on approximated Gaussian critical values with estimated correlation matrix}
\label{algo:boot2}
\begin{algorithmic}[1]
\State \textbf{Input:} $\tilde{X}_1,\hdots, \tilde{X}_{n_X}$,~$\lambda_{1,X}, \lambda_{1,X}', \lambda_{2,X}, \lambda_{2,X}'$,~$\overline{\eta}^X$ (all as above),~$\alpha \in (0, 1)$, and~$B \in \N$ with~$Q_{B, 1-\alpha}((1-1/B)^{1/d}) < B$.
\If{$\max(\eps_X, \eps_X') \geq 1/2$}
\State $\hat{\Pi}_{B,e} \gets [d]$
\Else
\State \textbf{Initialize:}~$k \gets 0$ and~$\hat{\Pi}_0 \gets[d]$.
\State $m(B, \alpha, d) \gets Q_{B, 1-\alpha}((1-1/B)^{1/d}) + 1$.
\State Draw~$z_1^{(k)}, \hdots, z_B^{(k)}$ independently from~$\mathsf{N}_{d}(0, \tilde{\Sigma}^{X}_{0})$ (and independent of $z_j^{(l)}$ for~$j = 1, \hdots, B$ and~$l < k$).
\State $\hat{c}_{1-\alpha, \hat{\Pi}_k}(\tilde{\Sigma}^{X}_{0}) \gets$ the~$m(B, \alpha, d)$th order statistic from~$\|z_1^{(k)}\|_{\infty, \hat{\Pi}_k}, \hdots, \|z_B^{(k)}\|_{\infty, \hat{\Pi}_k}$.
\While{$\cbr[1]{j\in \hat{\Pi}_k: |S^X_{W,j}|\leq \hat{c}_{1-\alpha, \hat{\Pi}_k}(\tilde{\Sigma}^{X}_{0})} \subsetneqq \hat{\Pi}_k$}  
\State $k \gets k+1$
\State $\hat{\Pi}_k \gets \cbr[1]{j\in \hat{\Pi}_{k-1}: |S^X_{W,j}|\leq \hat{c}_{1-\alpha, \hat{\Pi}_{k-1}}\tilde{\Sigma}^{X}_{0})}$
\If{$\hat{\Pi}_k = \emptyset$} 
\State \textbf{break}
\EndIf
\State repeat steps 7 \& 8 
\EndWhile
\EndIf
\State \textbf{Return:} $\hat{\Pi}_{B,e}  \gets \hat{\Pi}_k$.
\end{algorithmic}
\end{algorithm}

\begin{remark}\label{rem:mem}
We note that instead of generating~$d$-dimensional Gaussian random vectors in Step 7 of the while loop of Algorithm~\ref{algo:boot2}, one may (in a practical implementation of that algorithm) only generate the coordinates with indices in~$\hat{\Pi}_k$. This is preferable in terms of memory once~$\hat{\Pi}_k$ is (substantially) smaller than~$[d]$. The actual sampling step may incorporate the representation from Remark~\ref{rem:gramrep}, i.e., may be Gaussian multiplier bootstrap based.
\end{remark}

\begin{remark}\label{rem:empord}
Note that instead of working with an (infeasible) quantile (or ignoring that issue and working with the same quantile but from a Monte Carlo sample), Algorithm~\ref{algo:boot2} is based on the estimate~$\hat{c}_{1-\alpha, \hat{\Pi}_k}(\tilde{\Sigma}^{X}_{0})$, which is an order statistic at the order
\begin{equation}\label{eqn:mdef}
m(B, \alpha, d) := Q_{B, 1-\alpha}((1-1/B)^{1/d}) + 1.
\end{equation}
For illustration, we tabulate~$m(B, \alpha, d)/B$ for some a practically relevant ranges of~$\alpha \in \{0.1, 0.05, 0.01\}$,~$B = \{500, 1{,}000, 5{,}000, 10{,}000, 20{,}000\}$, and~$d \in \{100, 1{,}000, 10{,}000\}$ in Tables~\ref{tab:alpha_B}-\ref{tab:alpha_B3}, from which it can be seen that those values (i)~exceed~$(1-\alpha)$ (i.e.,~$\hat{c}_{1-\alpha, \hat{\Pi}_k}(\tilde{\Sigma}^{X}_{0})$ overestimates the target~$c_{1-\alpha, \hat{\Pi}_k}(\tilde{\Sigma}^{X}_{0})$), and (ii)~get close to~$(1-\alpha)$ for~$B$ large. For the ranges considered, the values~$B = 5,000$ or~$B = 10,000$ appear to be a good compromise between minimizing computational burden and achieving~$m(B, \alpha, d)/B \approx 1-\alpha$ and~$1/B$ (cf.~the upper bound in~\eqref{eqn:fwerB2} below) being fairly negligible relative to~$\alpha$.
\begin{table}[ht]
\centering
\begin{tabular}{lccccc}
\toprule
$1-\alpha$ & $B=500$ & $B=1,000$ & $B=5,000$ & $B=10{,}000$ & $B=20{,}000$ \\
\midrule
0.9 & 0.952 & 0.939 & 0.919 & 0.914 & 0.910 \\ 
0.95 & 0.986 & 0.978 & 0.964 & 0.960 & 0.957 \\ 
0.99 &  NA & NA & 0.996 & 0.995 & 0.993 \\ 

\bottomrule
\end{tabular}
\caption{$m(B, \alpha, 100)/B$ for different values of $\alpha$ and $B$; violation of~$Q_{B, 1-\alpha}((1-1/B)^{1/d}) < B$ results in NA.}
\label{tab:alpha_B}
\end{table}

\begin{table}[ht]
\centering
\begin{tabular}{lccccc}
\toprule
$1-\alpha$ & $B=500$ & $B=1,000$ & $B=5,000$ & $B=10{,}000$ & $B=20{,}000$ \\
\midrule
0.9 & 0.958 & 0.943 & 0.921 & 0.915 & 0.911 \\ 
0.95 & 0.990 & 0.980 & 0.965 & 0.961 & 0.958 \\ 
0.99 &  NA & NA & 0.996 & 0.995 & 0.994 \\ 

\hline
\end{tabular}
\caption{$m(B, \alpha, 1,000)/B$ for different values of $\alpha$ and $B$; violation of~$Q_{B, 1-\alpha}((1-1/B)^{1/d}) < B$ results in NA.}
\label{tab:alpha_B2}
\end{table}

\begin{table}[ht]
\centering
\begin{tabular}{lccccc}
\toprule
$1-\alpha$ & $B=500$ & $B=1,000$ & $B=5,000$ & $B=10{,}000$ & $B=20{,}000$ \\
\midrule
0.9 & 0.962 & 0.947 & 0.923 & 0.917 & 0.912 \\ 
0.95 & 0.992 & 0.983 & 0.966 & 0.962 & 0.959 \\ 
0.99 & NA & NA & 0.997 & 0.995 & 0.994 \\  

\hline
\end{tabular}
\caption{$m(B, \alpha, 10{,}000)/B$ for different values of $\alpha$ and $B$; violation of~$Q_{B, 1-\alpha}((1-1/B)^{1/d}) < B$ results in NA.}
\label{tab:alpha_B3}
\end{table}

\end{remark}

Theorem~\ref{thm:mainthB2} establishes, for Algorithm~\ref{algo:boot2}, essentially the same guarantees as that established for Algorithm~\ref{algo:boot} in Theorem~\ref{thm:mainthB}. In the upper bound below, there is an additional term relating to the bootstrap sample size (and thus the accuracy of the empirical quantiles used in its construction), which in practice one typically chooses large (e.g.,~$B = 5,000$ or~$B = 10,000$, cf.~Remark~\ref{rem:empord}).

\begin{theorem}\label{thm:mainthB2}
Let Assumption~\ref{ass:setting} be satisfied. Fix~$\alpha \in (0, 1)$, and choose a natural number~$B \geq 2$ such that~$Q_{B, 1-\alpha}((1-1/B)^{1/d}) < B$. Then,  
\begin{equation}\label{eqn:fwerB2}
\P\left(J(\mu^X) \not \subseteq \hat{\Pi}_{B,e}   \right) \leq \alpha + B^{-1} + C \times \left( \mathfrak{A}^X_{n_X} + \mathfrak{B}^X_{n_X}  + \mathfrak{C}^X_{n_X} \right),
\end{equation}
where $C$ is a constant depending only on~$b^X_1,b^X_2,\lambda_{1, X},\lambda_{2, X},\lambda_{1, X}',\lambda_{2, X}'$ and~$m_X$. The upper bound in~\eqref{eqn:fwerB2} converges to~$\alpha$ in any asymptotic regime satisfying Assumption~\ref{ass:asyX} under the additional condition that~$B \to \infty$.
\end{theorem}

The proof combines the idea of the proof of Theorem~\ref{thm:mainthB} with Lemma~\ref{lem:auxB}, which delivers that the order statistics in Algorithm~\ref{algo:boot2} (uniformly) dominate the critical values in Algorithm~\ref{algo:boot} with probability~$1-1/B$.

\section{Two-sample multiple testing under contamination}\label{sec:ts}

\subsection{Observational setting: the two-sample case}\label{sec:tsobs}

In addition to the random vectors~$X_1, \hdots, X_{n_X}$ and the contaminated sample~$\tilde X_1, \hdots, \tilde X_{n_X}$ introduced in Section~\ref{sec:osobs}, cf.~also Assumption~\ref{ass:setting}, we let~$Y_1, \hdots, Y_{n_Y}$ be~$d$-dimensional i.i.d.~random vectors with expectation~$\mu^Y := \E(Y_1)$ and covariance matrix~$\Sigma^Y := \E((Y_1 - \mu^Y)(Y_1 - \mu^Y)')$. We suppose that the statistician, in addition to observing~$\tilde X_1, \hdots, \tilde X_{n_X}$, also observes the contaminated ``second sample''~$\tilde{Y}_{1},\hdots,\tilde{Y}_{n_Y}$ of~$d$-dimensional random vectors satisfying
\begin{equation}\label{eq:contamfracY}
\left|\cbr[1]{i\in\cbr[0]{1,\hdots,n_Y}:\tilde{Y}_{i}\neq Y_{i}}\right|
\leq
\overline{\eta}^Y n_Y.
\end{equation}
The quantity~$\overline{\eta}^Y  \in[0,1/2)$ denotes a non-random upper bound on the fraction of contaminated observations among the~$Y_1, \hdots, Y_{n_Y}$. Note that we do \emph{not} assume~$\overline{\eta}^X = \overline{\eta}^Y$ (although this is not ruled out, of course). \emph{We assume throughout that~$n_Y > 3$, recall that~$d \geq 2$ is assumed}, and summarize the observational setup concerning the random vectors introduced above, together with some moment assumptions, as follows. Note that the following assumption just replaces~``$X$'' by~``$Y$'' at every occurrence in Assumption~\ref{ass:setting}, but we re-state the assumption in the following version for completeness.
\begin{assumption}\label{ass:settingY}
The~$Y_1,\hdots,Y_{n_Y}$ are i.i.d.~random vectors in~$\R^d$ with~$\E |Y_{1,j}|^{m_Y}<\infty$ for some~$m_Y\in(2,\infty)$ and all~$j=1,\hdots,d$,~$\mu^Y := \E(Y_1)$, and~$\Sigma^Y := \E((Y_1 - \mu^Y)(Y_1 - \mu^Y)')$. There exist strictly positive constants~$b^Y_1$ and~$b^Y_2$ such that~$\min_{j=1,\hdots,d}\sigma^Y_{2,j}\geq b^Y_1$ and $\sigma^Y_{m_Y}:=\max_{j=1,\hdots,d}\sigma^Y_{m_Y,j}\leq b^Y_2$. The actually observed adversarially contaminated random vectors (in~$\R^d$) are denoted~$\tilde{Y}_1,\hdots,\tilde{Y}_{n_Y}$ and satisfy~\eqref{eq:contamfracY}.	
\end{assumption}
At several instances, we also assume that the two samples available to the statistician are independent.
\begin{assumption}\label{ass:indep}
The samples~$(\tilde{X}_1, \hdots, \tilde{X}_{n_X})$ and~$(\tilde{Y}_1, \hdots, \tilde{Y}_{n_Y})$ are independent.
\end{assumption}
We note that the dimension~$d$ is common between the two samples. The reason is that we are interested in testing hypotheses on the underlying population expectations, which is only reasonable if the dimensions in the two samples coincide (otherwise, one may, of course, delete uncommon coordinates to start with). More specifically, in the adversarially contaminated two-sample setting described above, we consider the multiple testing problem
\begin{equation}\label{eq:testingproblemts}
\mathsf{H}_{0,j}: \mu^X_j = \mu^Y_j \qquad\text{vs.}\qquad \mathsf{H}_{1, j}: \mu^X_j \neq \mu^Y_j, \qquad \text{ for } j = 1, \hdots, d.
\end{equation}
Note that it is \emph{not} assumed that the covariance matrices~$\Sigma^X$ and~$\Sigma^Y$ coincide. We take care of this possibility in covariance estimation, where, in case it is \emph{known} that~$\Sigma^X = \Sigma^Y$, we recommend using a \emph{pooled} covariance estimator, whereas in case it is unknown whether or not~$\Sigma^X = \Sigma^Y$ we recommend another estimator.

Denoting~$\Delta := \mu^X - \mu^Y$ the multiple testing problem~\eqref{eq:testingproblemts} can  equivalently be stated in the following, somewhat more convenient, form:
\begin{equation}\label{eq:testingproblemts2}
\mathsf{H}_{0,j}: \Delta_j = 0 \qquad\text{vs.}\qquad \mathsf{H}_{1, j}: \Delta_j \neq 0, \qquad \text{ for } j = 1, \hdots, d.
\end{equation}

We denote~$J(\Delta) := \{j \in [d] : \Delta_j = 0\}$.

Similar to what was done in Section~\ref{sec:onesample}, in addition to our finite sample bounds, which mainly rely on Assumptions~\ref{ass:setting},~\ref{ass:settingY}, and~\ref{ass:indep}, we also formulate some asymptotic convergence statements in the following regime:
\begin{assumption}\label{ass:asyXY}
It holds that~$\min(n_X, n_Y) \to \infty$, and~$d = d(n_X, n_Y)$,~$\overline{\eta}_X = \overline{\eta}_X(n_X)$, and~$\overline{\eta}_Y = \overline{\eta}_Y(n_Y)$ satisfy $$\sqrt{n_X\log(d)}\overline{\eta}_X^{1-\frac{1}{m_X}} \vee \sqrt{n_Y\log(d)}\overline{\eta}_Y^{1-\frac{1}{m_Y}} \vee \log(d)/n_X^{\frac{m_X-2}{5m_X-2}} \vee \log(d)/n_Y^{\frac{m_Y-2}{5m_Y-2}} \to 0.$$
\end{assumption}

\subsection{Gaussian approximations and covariance estimation in the two-sample case and related quantiles}

In the two-sample case, our multiple testing procedures are based on winsorized estimators for the unknown population expectations and covariance matrices that are first computed in each of the two samples separately, and are then suitably combined. The estimators for the second sample are obtained in the same way as those in the one-sample case. More specifically, to obtain the estimators we do the following:
\begin{enumerate}
\item We define~$\eps_Y$ and~$\eps_Y'$ analogously to~$\eps_X$ and~$\eps_X'$ (which were defined in Equation~\eqref{eq:epsfamx}), but now based on~$\lambda_{1, Y}$ and~$\lambda_{1, Y}'$ in $(1, \infty)$ as well as~$\lambda_{2, Y}$ and~$\lambda_{2, Y}'$ in~$(0, \infty)$, and replacing~$n_X$ and~$\overline{\eta}^X$ by~$n_Y$ and~$\overline{\eta}^Y$, respectively. 
\item We define~$S^{\dagger, Y}_{W}$, $S^{\dagger, Y}_{W, S}$, and~$\tilde{\Sigma}^Y$ (including the variance estimators~$\tilde{\sigma}_{j}^Y$) and the corresponding winsorization points~$\hat{\alpha}_j^Y$,~$\hat{\beta}_j^Y$,~$\hat{a}_j^Y$,~$\hat{b}_j^Y$, analogously to~$S^{\dagger, X}_{W}$,~$S^{\dagger, X}_{W, S}$, and~$\tilde{\Sigma}^X$ (including the variance estimators~$\tilde{\sigma}_{j}^X$) and the corresponding winsorization points~$\hat{\alpha}_j^X$, $\hat{\beta}_j^X$,~$\hat{a}_j^X$,~$\hat{b}_j^X$, which were defined in~\eqref{eqn:winsp}-\eqref{eq:tildeSigma}, but now based on the quantities~$\tilde{Y}_1, \hdots, \tilde{Y}_{n_Y}$,~$\eps_Y$,~$\eps_Y'$,~$\overline{\eta}^Y$, and~$n_Y$.
\end{enumerate}
Concerning the choice of the tuning parameters~$\lambda_{1, Y}$ and~$\lambda_{1, Y}'$ in $(1, \infty)$ as well as~$\lambda_{2, Y}$ and~$\lambda_{2, Y}'$ in~$(0, \infty)$, we make the same practical recommendations as in Remark~\ref{rem:tpch}. 

Denoting~$$w_X := \sqrt{n_X/(n_X + n_Y)} > 0 \quad \text{ and } \quad w_Y := \sqrt{n_Y/(n_X + n_Y)} > 0,$$ we first define our winsorized estimator, and the corresponding centered version, as
\begin{equation}\label{eqn:sdagdelt}
S^{\dagger, \Delta}_{W} :=  w_Y S^{\dagger, X}_{W} - w_XS^{\dagger, Y}_{W} \quad \text{ and } \quad S^{\dagger, \Delta}_{W, S} := w_Y S^{\dagger, X}_{W, S} - w_X S^{\dagger, Y}_{W, S}.
\end{equation}
Note that~$w_X^2 + w_Y^2 = 1$, and that~$S^{\dagger, \Delta}_{W} - S^{\dagger, \Delta}_{W, S}$ is proportional to~$\Delta$, i.e., 
\begin{equation}\label{eqn:equiv}
S^{\dagger, \Delta}_{W} - S^{\dagger, \Delta}_{W, S} = w_{XY} \Delta, \quad \text{ for } \quad w_{XY} := \sqrt{(n_X n_Y)/(n_X + n_Y)}.
\end{equation}
This proportionality condition is fundamental, because~$S^{\dagger, \Delta}_{W, S}$ (for which limiting results will be developed below) is of course \emph{infeasible} (it incorporates~$\mu^X$ and~$\mu^Y$, which are unknown), whereas~\eqref{eqn:equiv} guarantees that under the null hypothesis~$\mathsf{H}_{0, j}$ it holds that the feasible~$S^{\dagger, \Delta}_{W, j} = S^{\dagger, \Delta}_{W, S, j}$. This relation also carries over to the properly normalized test statistics,~cf.~\eqref{eqn:coincts} below, cf.~Lemma~\ref{lem:coincts}.
\begin{remark}\label{rem:mxmy}
We here comment on the weights~$w_Y$ and~$w_X$ showing up in the above definitions, and which were not needed in the one-sample case. The properties (i)~positivity of~$w_X$ and~$w_Y$, (ii)~$w_X^2 + w_Y^2 = 1$, and (iii)~$S^{\dagger, \Delta}_{W, S} - S^{\dagger, \Delta}_{W}$ being proportional to~$\Delta$, i.e., $w_Y \sqrt{n_X} = w_X \sqrt{n_Y}$, together imply the particular form of the weights~$w_X$ and~$w_Y$ chosen above. While the positivity and proportionality conditions are important in the context of testing the hypotheses in~\eqref{eq:testingproblemts2}, the normalization condition~$w_X^2 + w_Y^2 = 1$ is, of course, completely arbitrary (although notationally convenient in the context of covariance estimation). It is reassuring that the choice of normalization is irrelevant for our purpose, as we shall base our testing procedures on \emph{normalized} test statistics, which are invariant to that choice.\footnote{Other common normalization conditions are: (1)~``$w_X^2 + w_Y^2 = n^{-1}_X + n^{-1}_Y$'', which under (i) and (ii) would result in the weights~``$w_X = n_Y^{-1/2}$'' and~``$w_Y = n_X^{-1/2}$''; or (2)~``$w_X^2 = 1$'', which under (i) and (ii) would result in the weights~``$w_Y = \sqrt{n_Y/n_X}$ and~$w_X = 1$''. Both normalizations have been frequently used (at least implicitly) in the definitions of two-sample test statistics for high-dimensional mean testing based on \emph{non-robust} location estimators, cf., e.g., Section 4 of~\cite{huang2022overview}. We have normalized the weights in a way that is most convenient for our purpose in terms of notation.}
\end{remark}

\subsubsection{Covariance estimation}\label{sec:covestts}

As in the one-sample case, we do not base our multiple testing procedures on the (non-normalized) winsorized means~$S^{\dagger, \Delta}_{W}$ directly. The reason, again, is that using a test statistic that is not normalized could be detrimental in situations where the variances of the coordinates differ substantially. To normalize the coordinates of~$S^{\dagger, \Delta}_{W}$, we estimate (cf.~Theorem~\ref{thm:HDGaussts}, where the subsequent quantity appears as the covariance matrix of an approximating Gaussian) $$\Sigma^{\Delta} := w^2_Y \Sigma^X + w^2_X\Sigma^Y$$ via
\begin{equation}\label{eqn:covesttspo}
\tilde{\Sigma}^{\Delta} := \begin{cases} w^2_Y \tilde{\Sigma}^X + w^2_X\tilde{\Sigma}^Y, & \text{ in case it is possible that } \Sigma^X \neq \Sigma^Y, \\
w_X^2 \tilde{\Sigma}^X + w_Y^2 \tilde{\Sigma}^Y, & \text{ in case it is known that } \Sigma^X = \Sigma^Y.
\end{cases}
\end{equation}
This choice requires some discussion: on the one hand, in case it is known that~$\Sigma^X = \Sigma^Y$ (which, due to our normalization of the weights, cf.~Remark~\ref{rem:mxmy}, then also coincide with~$\Sigma^{\Delta}$), we would use a \emph{pooled} covariance estimator (cov.~est.). Recalling~\eqref{eq:tildeSigma}, this pooled estimator equals
\begin{align*}
\tilde{\Sigma}^{\Delta}_{j,k} 
&=
w_X^2 \times n_X^{-1}\sum_{i=1}^{n_X}\sbr[1]{\phi_{\hat{a}^X_j,\hat{b}^X_j}(\tilde{X}_{i,j})- \tilde{\mu}^X_{j}}\sbr[1]{\phi_{\hat{a}^X_k,\hat{b}^X_k}(\tilde{X}_{i,k})-\tilde{\mu}^X_{k}} \\ & \hspace{3cm} + 
w_Y^2 \times n_Y^{-1}\sum_{i=1}^{n_Y}\sbr[1]{\phi_{\hat{a}^Y_j,\hat{b}^Y_j}(\tilde{Y}_{i,j})-\tilde{\mu}^Y_{j}}\sbr[1]{\phi_{\hat{a}^Y_k,\hat{b}^Y_k}(\tilde{Y}_{i,k})-\tilde{\mu}^Y_{k}} \\
&= (n_X + n_Y)^{-1} \sum_{V \in \{X, Y\}} \sum_{i = 1}^{n_V} 
\sbr[1]{\phi_{\hat{a}^V_j,\hat{b}^V_j}(\tilde{V}_{i,j})-\tilde{\mu}^V_{j}}\sbr[1]{\phi_{\hat{a}^V_k,\hat{b}^V_k}(\tilde{V}_{i,k})-\tilde{\mu}^V_{k}}.
\end{align*}
On the other hand, in situations where it is unknown whether the covariance matrices in the two samples coincide or not, we do \emph{not} use a pooled estimator. We then weight the two summands in~$\tilde{\Sigma}^{\Delta}$ in exact correspondence to the weights used for location estimation~\eqref{eqn:sdagdelt}, cf.~\eqref{eqn:covesttspo}. [We note that in the special case where~$n_X = n_Y$, so that~$w_X = w_Y$, the two expressions in the two cases in~\eqref{eqn:covesttspo} coincide, regardless of whether~$\Sigma^X = \Sigma^Y$ or~$\Sigma^X \neq \Sigma^Y$.] Let~$\tilde{\sigma}_{j}^{\Delta}$ denote the square root of the~$j$-th diagonal entry of~$\tilde{\Sigma}^{\Delta}$, and denote by~$\sigma_{2,j}^{\Delta}$ the square root of the~$j$-th diagonal element of~$\Sigma^{\Delta}.$

The covariance estimator~$\tilde{\Sigma}^{\Delta}$ defined above satisfies the following precision guarantees. Note that the independence condition formulated in Assumption~\ref{ass:indep} is \emph{not} needed here (but also note that in case it is violated, the quantity~$\Sigma^{\Delta}$, as defined above, may no longer be a relevant target!).
\begin{proposition}\label{prop:covestimGRAM2s}
Let Assumptions~\ref{ass:setting} and~\ref{ass:settingY} be satisfied.  If 
\begin{equation}\label{eqn:epscond}
\max_{V \in \{X, Y\}} \left[2\eps'_V +\frac{\log(d^2n_V)}{n_V}+\sqrt{\del[2]{\frac{\log(d^2n_V)}{n_V}}^2+4\frac{\log(d^2n_V)}{n_V}\eps_V'} \right]<1,
\end{equation}
then, for a constant~$C$ depending only on~$b^V_2,\lambda_{1,V}',\lambda_{2,V}'$ and~$m_V$ for~$V \in \{X, Y\}$, 
\begin{equation}\label{eq:covestimGramts}
\P\del[4]{\max_{1\leq j,k\leq d}\envert[1]{\tilde{\Sigma}^{\Delta}_{j,k}-\Sigma^{\Delta}_{j,k}}> C \mathfrak{d}^{\Delta}  }
\leq  24 \times  \frac{n_Y + n_X}{n_Xn_Y},
\end{equation}	
where~
\begin{equation}\label{eqn:ddefts}
\mathfrak{d}^{\Delta} 
:= 
\begin{cases}
w_Y^2\mathfrak{d}_{n_X}^X + w_X^2 \mathfrak{d}_{n_Y}^Y, & \text{ if the non-pooled cov.~est. is used}, \\
w_X^2 \mathfrak{d}_{n_X}^X + w_Y^2 \mathfrak{d}_{n_Y}^Y, & \text{ if~$\Sigma^X = \Sigma^Y$ and the pooled cov.~est. is used.}
\end{cases}
\end{equation}
\end{proposition}
Given our covariance estimator, we can define the normalized test statistic~$S^{\Delta}_{W}$ our multiple testing procedures are based on, together with its centered version~$S^{\Delta}_{W, S}$, via their coordinates~$j = 1, \hdots, d$ as follows
\begin{equation}\label{eqn:sdagdeltN}
S^{\Delta}_{W, j} := \frac{S^{\dagger, \Delta}_{W, j}}{\tilde{\sigma}_{j}^{\Delta}} = \frac{w_Y S^{\dagger, X}_{W, j} - w_X S^{\dagger, Y}_{W, j}}{\tilde{\sigma}_{j}^{\Delta}} ~ \text{ and } ~
S^{\Delta}_{W, S, j} := \frac{S^{\dagger, \Delta}_{W, S, j}}{\tilde{\sigma}_{j}^{\Delta}} = \frac{w_Y S^{\dagger, X}_{W, S, j} - w_X S^{\dagger, Y}_{W, S, j}}{\tilde{\sigma}_{j}^{\Delta}},
\end{equation}
and we set~$S^{\Delta}_{W, j} = 0$ in case~$\tilde{\sigma}_{j}^{\Delta} = 0$, but leave~$S^{\Delta}_{W, S}$ undefined in that case. 

The following lemma is an immediate consequence of the \emph{proportionality requirement} on the weights~$w_X$ and~$w_Y$ used in the construction of the two-sample winsorized statistics; cf.~the discussion around~\eqref{eqn:equiv} and Remark~\ref{rem:mxmy}. It plays a subtle, but crucial role when applying the Gaussian approximation results (from the subsequent section) to obtain guarantees for our testing procedures: It relates the feasible, but non-centered test statistic~$S^{\Delta}_{W}$, which our testing procedures will be based on, to the centered, but \emph{infeasible} one~$S^{\Delta}_{W, S}$, for which we shall obtain Gaussian approximation guarantees. Lemma~\ref{lem:coincts} therefore is a tool that allows us to ``carry over'' approximations from the centered to the non-centered case, at least for coordinates~$j$ where~$\Delta_j = 0$. [A related statement underlies the results in the one-sample case, but that result is more obvious as no weighting satisfying particular requirements need to be introduced (cf.~Equation~\eqref{eqn:coinc} in Appendix~\ref{app:aux}).]
\begin{lemma}\label{lem:coincts}
If~$j \in J(\Delta) = \{j \in [d] : \Delta_j = 0\}$ and~$\tilde{\sigma}^{{\Delta}}_{j} \neq 0$, it holds that 
\begin{equation}\label{eqn:coincts}
S^{\Delta}_{W,S,j} = S^{\Delta}_{W,j}.
\end{equation}
\end{lemma}
\subsubsection{Two-sample Gaussian approximation results}

We next establish a Gaussian approximation inequality for the distribution of~$S^{\dagger, \Delta}_{W}$. It constitutes a two-sample analogue of Theorem 2.1 in \cite{kphdapprox}, cf.~Theorem~\ref{thm:HDGauss} in Appendix~\ref{app:aux}, which built the basis for the theoretical results in Section~\ref{sec:onesample}. We define the quantity~$\mathfrak{A}_{n_Y}^Y$ analogously to~$\mathfrak{A}_{n_X}^X$, which was defined in~\eqref{eqn:abcdef}.
\begin{theorem}\label{thm:HDGaussts}
Let Assumptions~\ref{ass:setting},~\ref{ass:settingY}, and~\ref{ass:indep} be satisfied.  If~$\eps_X,\eps_Y \in(0,1/2)$, then, for~$Z\sim\mathsf{N}_d(0,\Sigma^{\Delta})$ and~$C^{\Delta}$ a constant depending only on~$b^V_1,b^V_2$, $\lambda_{1,V},\lambda_{2,V}$ and~$m_V$ for~$V \in \{X, Y\}$,
\begin{equation}\label{eq:HDGaussts}
\rho^{\Delta}_{W}
:=
\sup_{H\in\mc{H}}\envert[2]{\P\del[1]{S^{\dagger, \Delta}_{W, S} \in H}-\P\del[1]{Z\in H}}
\leq
C^{\Delta} \times \left(\mathfrak{A}^X_{n_X} + \mathfrak{A}^Y_{n_Y} \right).
\end{equation}
In particular,~$\rho^{\Delta}_{W} \to 0$ in any asymptotic regime satisfying Assumption~\ref{ass:asyXY}.
\end{theorem}
The proof of Theorem~\ref{thm:HDGaussts} is based on a conditioning argument, exploiting in particular Assumption~\ref{ass:indep}, allowing us to carry over the one-sample Gaussian approximation results from Theorem 2.1 in \cite{kphdapprox} to the present two-sample setting; cf.~also Remark~\ref{rem:carryoverGts}. Theorem~\ref{thm:HDGaussts} shows that (careful) winsorization allows one to establish Gaussian approximation results with exponentially increasing dimensionality despite weak moment assumptions and under adversarial contamination. 

Denoting the correlation matrix corresponding to~$\Sigma^{\Delta}$ by~$\Sigma_0^{\Delta}$, we now present a Gaussian approximation inequality for the distribution of~$S^{\Delta}_{W, S}$. This result is our main tool to derive theoretical guarantees for our testing procedures in the two-sample case, and can be viewed as a two-sample version of Theorem 4.1 in \cite{kphdapprox}, cf.~Theorem~\ref{thm:HDGauss_studentized} in Appendix~\ref{app:aux}. Its proof is based on Proposition~\ref{prop:covestimGRAM2s} and Theorem~\ref{thm:HDGaussts}. We here define, recall the definition of~$\mathfrak{d}^{\Delta}$ from~\eqref{eqn:ddefts},
\begin{equation}\label{eqn:Bdeltadef}
\mathfrak{B}^{\Delta} := \sqrt{\log(d)\log(d\min(n_X, n_Y))} \times \mathfrak{d}^{\Delta}.
\end{equation}
[In the formulation of Theorem~\ref{thm:HDGauss_studentizedts} it is understood, cf.~also the definition of~$\tilde{\Sigma}^{\Delta}$ in~\eqref{eqn:covesttspo}, that the pooled covariance estimator is used only in case it is known and correct that~$\Sigma^X = \Sigma^Y$.]

\begin{theorem}\label{thm:HDGauss_studentizedts}
Let Assumptions~\ref{ass:setting},~\ref{ass:settingY}, and~\ref{ass:indep} be satisfied.  If~$\eps_X,\eps_X'\in(0,1/2)$ and~$\eps_Y,\eps_Y'\in(0,1/2)$, then, for~$Z'\sim\mathsf{N}_d(0,\Sigma_0^{\Delta})$,
\begin{equation}\label{eq:HDGaussstudentizedts}
\rho^{\Delta}_{W,S}
:=
\sup_{H\in\mc{H}}\envert[2]{\P\del[1]{S^{\Delta}_{W,S}\in H}-\P\del[1]{Z'\in H}}
\leq
C_S^{\Delta} \times \left(\mathfrak{A}^X_{n_X} + \mathfrak{A}^Y_{n_Y}+\mathfrak{B}^{\Delta}\right),
\end{equation}
where~$C_S^{\Delta}$ is a constant depending only on~$b^V_1,b^V_2,\lambda_{1,V},\lambda_{2,V},\lambda_{1,V}',\lambda_{2,V}'$ and~$m_V$ for~$V \in \{X, Y\}$. In particular, the upper bound in~\eqref{eq:HDGaussstudentizedts}, and hence~$\rho^{\Delta}_{W, S}$, converges to~$0$ in any asymptotic regime satisfying Assumption~\ref{ass:asyXY}.
\end{theorem}

It is important to emphasize that the approximation statements in~\eqref{eq:HDGaussts} and~\eqref{eq:HDGaussstudentizedts}, and the corresponding limiting statements under regimes satisfying Assumption~\ref{ass:asyXY}, hold (i)~despite adversarial contamination, and that (ii)~they deliver~$\rho^{\Delta}_{W} \to 0$ and~$\rho^{\Delta}_{W,S} \to 0$ in settings where the dimension~$d$ can increase exponentially in sample size, while only imposing~$m_X > 2$ and~$m_Y > 2$ moments to exist (cf.~also the discussion in \cite{kphdapprox} for the one-sample case). Neither of statements (i) and (ii) do not necessarily hold for standard sample means, which are highly non-robust, and typically require sub-exponential tail conditions to guarantee (ii), cf.~\cite{zhangwu2017} and~\cite{kock2024remark}.

\subsubsection{Bootstrap quantiles for the two-sample case}\label{sec:bootqts}

To define the (conditional) quantile functions for the two-sample setting, we obtain the correlation matrix~$\tilde{\Sigma}^{\Delta}_0$, say, corresponding to~$\tilde{\Sigma}^{\Delta}$ in the same way as we obtained~$\tilde{\Sigma}^X_0$ from~$\tilde{\Sigma}^X$ in Equation~\eqref{eqn:corrmatdef}. Regardless of whether a diagonal element of~$\tilde{\Sigma}^{\Delta}$ vanishes or not,~$\tilde{\Sigma}_0^{\Delta}$ is then always well defined and constitutes a correlation matrix. Analogously to the definitions in Section~\ref{sec:quants}, but now for the two-sample case, we denote, for a non-empty set~$\hat{A} = \hat{A}(\tilde{X}_1, \hdots, \tilde{X}_{n_X}, \tilde{Y}_1, \hdots, \tilde{Y}_{n_Y}) \subseteq [d]$ by~$[0, 1] \ni \beta \mapsto c_{\beta, \hat{A}}(\tilde{\Sigma}^{\Delta}_0)$ the (conditional) quantile function of the random variable (cf.~also the notation introduced in Footnote~\ref{foot:quantinv2})~$\max_{j \in \hat{A}} | (\tilde{\Sigma}_0^{\Delta, 1/2}Z)_j |$ (the quantile function is defined as~$0$ for~$\beta = 0$, and as~$\infty$ for~$\beta = 1$), which clearly satisfies the monotonicity property 
\begin{equation}\label{eq:MonCVBts}
c_{\cdot,\hat{A}_1}(\tilde{\Sigma}^{\Delta}_0) \leq c_{\cdot,\hat{A}_2}(\tilde{\Sigma}^{\Delta}_0) \quad \text{ whenever } \quad \emptyset \neq \hat{A}_1\subseteq \hat{A}_2\subseteq[d].
\end{equation}
We analogously define~$[0, 1] \ni \beta \mapsto c_{\beta, \hat{A}}(\Sigma^{\Delta}_0)$, which also satisfies a monotonicity relation analogous to~\eqref{eq:MonCVBts}, and is related to~$\beta \mapsto c_{\beta, \hat{A}}(\tilde{\Sigma}^{\Delta}_0)$ via Proposition~\ref{prop:BSAlwaysvalidts}, a technical device used in establishing guarantees for the bootstrap-based multiple testing procedures in the two-sample case. 

\subsection{Testing the global null hypothesis in the two-sample case}\label{sec:globnull2s}

Size and power properties of supremum norm-based tests for the one-sample global null defined through the hypotheses in~\eqref{eq:testingproblem} based on winsorized statistics have been thoroughly investigated in \cite{kprobfree}. Two-sample versions of such (winsorization-based) tests have neither been introduced nor investigated yet. Equipped with Theorem~\ref{thm:HDGauss_studentizedts} above, together with Proposition~\ref{prop:BSAlwaysvalidts} in Appendix~\ref{app:proofsts}, we now briefly consider the high-dimensional \emph{global testing problem} defined by the hypotheses in~\eqref{eq:testingproblemts2}. This high-dimensional two-sample global testing problem, based on non-robust statistics, has recently been considered in \cite{xue2020distribution}. 

We shall establish an upper bound on the null rejection probability of the \emph{test} that rejects if and only if
\begin{equation}\label{eqn:tsglobG}
\|S^{\Delta}_{W}\|_{\infty} > c_{1-\alpha, [d]},
\end{equation}
and the \emph{test} that rejects if and only if
\begin{equation}\label{eqn:tsglobB}
\|S^{\Delta}_{W}\|_{\infty} > c_{1-\alpha, [d]}(\tilde{\Sigma}^{\Delta}_0),
\end{equation}
particularly showing asymptotic size control under regimes satisfying Assumption~\ref{ass:asyXY}. 

As already pointed out in the context of Algorithms~\ref{algo:boot} and~\ref{algo:boot2}, a critical value such as~$c_{1-\alpha}(\tilde{\Sigma}^{\Delta}_0)$ above is not computable in closed form. Therefore, in practice, one has to determine it numerically (which is less of an issue in the context of global hypothesis testing, as the critical value only has to be determined once). To incorporate this into our analysis, we also provide a result for a test based on a numerically obtained quantile. To this end, \emph{similar} to what was done in Algorithm~\ref{algo:boot2}, conditionally on the data, let~$z_1, \hdots, z_B$ be independently drawn from~$\mathsf{N}_d(0, \tilde{\Sigma}^{\Delta}_0)$ and denote by~$$\hat{c}_{1-\alpha, [d]}(\tilde{\Sigma}^{\Delta}_0) := \text{the } (Q_{B, 1-\alpha}(1-1/B) + 1)\text{th order statistic from } \quad \|z_1\|_{\infty}, \hdots, \|z_B\|_{\infty}.$$ Note that we were working with~$m(B, \alpha, d)$ in Algorithm~\ref{algo:boot2}, whereas in the present context, where there is only one (as opposed to possibly~$d$ steps), one can work with a single bootstrap sample and (the corresponding)~$m(B, \alpha, 1)$. The result is as follows. [In the formulation of Theorem~\ref{thm:mainthGtsgl} it is understood, cf.~also the definition of~$\tilde{\Sigma}^{\Delta}$ in~\eqref{eqn:covesttspo}, that the pooled covariance estimator is used only in case it is known and correct that~$\Sigma^X = \Sigma^Y$.] To state the result, we introduce~$\mathfrak{C}^{\Delta} := \log(d) \times \sqrt{\mathfrak{d}^{\Delta}}$.
\begin{theorem}\label{thm:mainthGtsgl}
Let Assumptions~\ref{ass:setting},~\ref{ass:settingY}, and~\ref{ass:indep} be satisfied, and fix~$\alpha \in (0, 1)$. Under the condition that~$\Delta = 0$, and assuming that~$\eps_X,\eps_X'\in(0,1/2)$ and~$\eps_Y,\eps_Y'\in(0,1/2)$, it holds that
\begin{equation}\label{eqn:fwerGtsG}
\P\left(\|S^{\Delta}_{W}\|_{\infty} > c_{1-\alpha, [d]} \right) \leq \alpha  + C \times \left(\mathfrak{A}^X_{n_X} + \mathfrak{A}^Y_{n_Y}+\mathfrak{B}^{\Delta}\right),
\end{equation}
\begin{equation}\label{eqn:fwerGtsB}
\P\left(\|S^{\Delta}_{W}\|_{\infty} > c_{1-\alpha, [d]}(\tilde{\Sigma}^{\Delta}_0) \right) \leq \alpha  + C' \times \left(\mathfrak{A}^X_{n_X} + \mathfrak{A}^Y_{n_Y}+\mathfrak{B}^{\Delta} + \mathfrak{C}^{\Delta} \right);
\end{equation}
and, for any natural number~$B \geq 2$ such that~$Q_{B, 1-\alpha}(1-1/B)<B$, 
\begin{equation}\label{eqn:fwerGtsB2}
\P\left(\|S^{\Delta}_{W}\|_{\infty} > \hat{c}_{1-\alpha, [d]}(\tilde{\Sigma}^{\Delta}_0) \right) \leq \alpha  + B^{-1} + C'' \times \left(\mathfrak{A}^X_{n_X} + \mathfrak{A}^Y_{n_Y}+\mathfrak{B}^{\Delta} +\mathfrak{C}^{\Delta}\right),
\end{equation}
where~$C, C', C''$ are constants depending only on~$b^V_1,b^V_2,\lambda_{1, V},\lambda_{2, V},\lambda_{1, V}',\lambda_{2, V}'$ and~$m_V$ for~$V \in \{X, Y\}$. The upper bounds in~\eqref{eqn:fwerGtsG} and~\eqref{eqn:fwerGtsB} converge to~$\alpha$ in any asymptotic regime satisfying Assumption~\ref{ass:asyXY}, while the upper bound in~\eqref{eqn:fwerGtsB2} requires the additional condition that~$B \to \infty$ to converge to~$\alpha$.
\end{theorem}
The upper bound in~\eqref{eqn:fwerGtsG} is for test based on worst-case critical values obtained from an application of the Khatri-{\v{S}}id{\'a}k inequality. The corresponding test is hence conservative in case of notable correlations in~$\Sigma^{\Delta}_0$. The upper bound in~\eqref{eqn:fwerGtsB} adapts to that correlation matrix, but has the disadvantage that the critical value it is based on is not available in closed form, but needs to be determined approximately. The upper bound on the test in~\eqref{eqn:fwerGtsB2}, which is based on a (carefully chosen) empirical quantile from a bootstrap sample, explicitly carries a term that illustrates the effect of approximating the critical value. We recommend the test corresponding to~\eqref{eqn:fwerGtsB2} for practical use.

For illustration and ease of implementation in practice (although the quantity is immediately obtained in any reasonable statistical computing environment), we tabulate~$(Q_{B, 1-\alpha}(1-1/B) + 1)/B$ for some values for common~$B$ and~$\alpha$ (similar to what has been done in Tables~\ref{tab:alpha_B}-\ref{tab:alpha_B3} in Remark~\ref{rem:empord} in the one-sample multiple testing framework). The values are given in Table~\ref{tab:ght2} and illustrate that~$B = 10{,}000$ achieves~$(Q_{B, 1-\alpha}(1-1/B) + 1)/B \approx (1-\alpha)$ and~$B^{-1} \approx 0$ in practice, without imposing a too heavy  computational burden.

\begin{table}[ht]
\centering
\begin{tabular}{rrrrrr}	
	\hline
	& 500 & 1{,}000 & 5{,}000 & 10{,}000 & 20{,}000 \\ 
	\hline
	0.9 & 0.938 & 0.929 & 0.915 & 0.911 & 0.908 \\ 
	0.95 & 0.978 & 0.971 & 0.961 & 0.958 & 0.956 \\ 
	0.99 & NA & 0.999 & 0.995 & 0.994 & 0.993 \\ 
	\hline
\end{tabular}
\caption{$(Q_{B, 1-\alpha}(1-1/B) + 1)/B$ for different values of $\alpha$ and $B$; violation of~$Q_{B, \beta}(\gamma) < B$ results in NA.}
	\label{tab:ght2}
\end{table}

\subsection{Multiple testing procedures and theoretical guarantees}\label{sec:tsmt}

We now propose robust step-down procedures~$\hat{\Pi} = \hat{\Pi}(\tilde{X}_1, \hdots, \tilde{X}_{n_X}, \tilde{Y}_1, \hdots, \tilde{Y}_{n_Y}) \subseteq [d]$, say, to identify the true null hypotheses among~\eqref{eq:testingproblemts2}, i.e.,~$$J(\Delta) = \{j \in [d] : \Delta_j = 0\}.$$ As in the one-sample case, our main goal is to suggest procedures that asymptotically control the FWER, i.e., the probability~$\P(J(\Delta) \not \subseteq \hat{\Pi} )$ that at least one true null hypothesis is not contained in~$\hat{\Pi}$. The procedures we discuss here are two-sample versions of the procedures introduced for the one-sample case. Again, we state a computationally parsimonious procedure based on worst-case Gaussian critical values, and a procedure based on (multiplier) bootstrap quantiles, which comes with the advantage of being less conservative in case of notable departure of~$\Sigma^{\Delta}_0$ from the identity matrix. Furthermore, since the latter critical values need to be approximated in practice, we also discuss a readily implementable procedure that is based on explicit empirical quantiles from a bootstrap sample, and takes this approximation step into account in the corresponding upper bound. For ease of implementability and comparability we state the algorithms and theorems separately in analogy to the one-sample case, but we keep the discussion of the algorithms and results (which are similar to the one-sample case) to a minimum.

Our multiple testing procedure based on Gaussian critical values for the two-sample case is defined as follows.
\begin{algorithm}[H]
\caption{Step-down procedure for testing \eqref{eq:testingproblemts} based on Gaussian critical values with worst-case correlation matrix (i.e., under independence)}
\label{algo:gaussts}
\begin{algorithmic}[1]
\State \textbf{Input:} $\tilde{V}_1,\hdots, \tilde{V}_{n_V}$,~$\lambda_{1,V}, \lambda_{1,V}', \lambda_{2,V}, \lambda_{2,V}'$,~$\overline{\eta}^V$, for~$V \in \{X, Y\}$ (and in the respective ranges as required), and~$\alpha \in (0, 1)$.
\If{$\max(\eps_V, \eps_V': V \in \{X, Y\}) \geq 1/2$}
\State $\hat{\Pi}^{\Delta}_G \gets [d]$
\Else \State \textbf{Initialize:}~$k \gets 0$ and~$\hat{\Pi}_0 \gets[d]$

\While{$\hat{\Pi}_k \neq \emptyset$ and $\cbr[1]{j\in \hat{\Pi}_k: |S^{\Delta}_{W,j}|\leq c_{1-\alpha, \hat{\Pi}_k}} \subsetneqq \hat{\Pi}_k$}  \State $k \gets k+1$
\State $\hat{\Pi}_k \gets \cbr[1]{j\in \hat{\Pi}_{k-1}: |S^{\Delta}_{W,j}|\leq c_{1-\alpha, \hat{\Pi}_{k-1}}}$
\EndWhile
\EndIf
\State \textbf{Return:} $\hat{\Pi}^{\Delta}_G \gets \hat{\Pi}_k$.
\end{algorithmic}
\end{algorithm}

We also establish an upper bound on the FWER of $\hat{\Pi}^{\Delta}_G$. [In the formulation of Theorem~\ref{thm:mainthGts} it is understood, cf.~also the definition of~$\tilde{\Sigma}^{\Delta}$ in~\eqref{eqn:covesttspo}, that the pooled covariance estimator is used only in case it is known and correct that~$\Sigma^X = \Sigma^Y$.]
\begin{theorem}\label{thm:mainthGts}
Let Assumptions~\ref{ass:setting},~\ref{ass:settingY}, and~\ref{ass:indep} be satisfied, and fix~$\alpha \in (0, 1)$. Then,
\begin{equation}\label{eqn:fwerGts}
\P\left(J(\Delta) \not \subseteq \hat{\Pi}_G^{\Delta}  \right) \leq \alpha  + C \times \left(\mathfrak{A}^X_{n_X} + \mathfrak{A}^Y_{n_Y}+\mathfrak{B}^{\Delta}\right),
\end{equation}
where $C$ is a constant depending only on~$b^V_1,b^V_2,\lambda_{1, V},\lambda_{2, V},\lambda_{1, V}',\lambda_{2, V}'$ and~$m_V$ for~$V \in \{X, Y\}$. The upper bound in~\eqref{eqn:fwerGts} converges to~$\alpha$ in any asymptotic regime satisfying Assumption~\ref{ass:asyXY}.
\end{theorem}
We next discuss Algorithm~\ref{algo:bootts}, a version of Algorithm~\ref{algo:gaussts} that is based on an estimate of the correlation matrix. It adapts to the underlying correlation structure~$\Sigma_0^{\Delta}$, and can hence be less conservative. While the quantiles in Algorithm~\ref{algo:gaussts} come in closed form (cf.~\eqref{eqn:Gqcf}), the ones used in Algorithm~\ref{algo:bootts} do not. One typically has to determine those quantiles numerically, e.g., by a Gaussian multiplier based bootstrap (making use of the Gram representation pointed out in Remark~\ref{rem:gramrep}). As already announced above, we take these approximation steps into account more explicitly further below in the context of Algorithm~\ref{algo:bootts2}. 

\begin{algorithm}[H]
\caption{Step-down procedure for testing \eqref{eq:testingproblemts} based on Gaussian critical values with estimated correlation matrix}
\label{algo:bootts}
\begin{algorithmic}[1]
\State \textbf{Input:} $\tilde{V}_1,\hdots, \tilde{V}_{n_V}$,~$\lambda_{1,V}, \lambda_{1,V}', \lambda_{2,V}, \lambda_{2,V}'$,~$\overline{\eta}^V$, for~$V \in \{X, Y\}$ (and in the respective ranges as required), and~$\alpha \in (0, 1)$.
\If{$\max(\eps_V, \eps_V': V \in \{X, Y\}) \geq 1/2$}
\State $\hat{\Pi}_B^{\Delta} \gets [d]$
\Else
\State \textbf{Initialize:}~$k \gets 0$ and~$\hat{\Pi}_0 \gets[d]$
\While{$\hat{\Pi}_k \neq \emptyset$ and $\cbr[1]{j\in \hat{\Pi}_k: |S^{\Delta}_{W,j}|\leq c_{1-\alpha, \hat{\Pi}_k}(\tilde{\Sigma}^{\Delta}_{0})} \subsetneqq \hat{\Pi}_k$}  
\State $k \gets k+1$
\State $\hat{\Pi}_k \gets \cbr[1]{j\in \hat{\Pi}_{k-1}: |S^{\Delta}_{W,j}|\leq c_{1-\alpha, \hat{\Pi}_{k-1}}(\tilde{\Sigma}^{\Delta}_{0})}$
\EndWhile
\EndIf
\State \textbf{Return:} $\hat{\Pi}^{\Delta}_B \gets \hat{\Pi}_k$.
\end{algorithmic}
\end{algorithm}

We now establish an upper bound on the FWER of $\hat{\Pi}^{\Delta}_B$. [In the formulation of Theorem~\ref{thm:mainthBts} it is understood, cf.~also the definition of~$\tilde{\Sigma}^{\Delta}$ in~\eqref{eqn:covesttspo}, that the pooled covariance estimator is used only in case it is known and correct that~$\Sigma^X = \Sigma^Y$.]
\begin{theorem}\label{thm:mainthBts}
Let Assumptions~\ref{ass:setting},~\ref{ass:settingY}, and~\ref{ass:indep} be satisfied, and fix~$\alpha \in (0, 1)$. Then, 
\begin{equation}\label{eqn:fwerBts}
\P\left(J(\Delta) \not \subseteq \hat{\Pi}_B^{\Delta}  \right) \leq \alpha +  C \times \left(\mathfrak{A}^X_{n_X} + \mathfrak{A}^Y_{n_Y}+\mathfrak{B}^{\Delta}+\mathfrak{C}^{\Delta}\right),
\end{equation}
where $C$ is a constant depending only on~$b^V_1,b^V_2,\lambda_{1, V},\lambda_{2, V},\lambda_{1, V}',\lambda_{2, V}'$ and~$m_V$ for~$V \in \{X, Y\}$. The upper bound in~\eqref{eqn:fwerBts} converges to~$\alpha$ in any asymptotic regime satisfying Assumption~\ref{ass:asyXY}.
\end{theorem}
Similar to Algorithm~\ref{algo:boot2}, we next state and analyze a procedure that incorporates this approximation step explicitly in Algorithm~\ref{algo:bootts2}. We use the same notation as in the description of Algorithm~\ref{algo:boot2}, cf.~also Lemma~\ref{lem:auxB} together with its discussion in Remark~\ref{rem:empord} and the exemplary values tabulated there. That discussion also applies here, but is not repeated. The same applies to the discussion in Remark~\ref{rem:mem}.

\begin{algorithm}[H]
\caption{Step-down procedure for testing \eqref{eq:testingproblemts} based on approximated Gaussian critical values with estimated correlation matrix}
\label{algo:bootts2}
\begin{algorithmic}[1]
\State \textbf{Input:} $\tilde{V}_1,\hdots, \tilde{V}_{n_V}$,~$\lambda_{1,V}, \lambda_{1,V}', \lambda_{2,V}, \lambda_{2,V}'$,~$\overline{\eta}^V$, for~$V \in \{X, Y\}$ (and in the respective ranges as required),~$\alpha \in (0, 1)$, and~$B \in \N$ with~$Q_{B, 1-\alpha}((1-1/B)^{1/d}) < B$.
\If{$\max(\eps_V, \eps_V': V \in \{X, Y\}) \geq 1/2$}
\State $\hat{\Pi}_{B,e}^{\Delta} \gets [d]$
\Else
\State \textbf{Initialize:}~$k \gets 0$ and~$\hat{\Pi}_0 \gets[d]$
		\State $m(B, \alpha, d) \gets Q_{B, 1-\alpha}((1-1/B)^{1/d}) + 1$.
\State Draw~$z_1^{(k)}, \hdots, z_B^{(k)}$ independently from~$\mathsf{N}_{d}(0, \tilde{\Sigma}^{\Delta}_{0})$ (and independent of $z_j^{(l)}$ for~$j = 1, \hdots, B$ and~$l < k$).
\State $\hat{c}_{1-\alpha, \hat{\Pi}_k}(\tilde{\Sigma}^{\Delta}_{0}) \gets$ the~$m(B, \alpha, d)$th order statistic from~$\|z_1^{(k)}\|_{\infty, \hat{\Pi}_k}, \hdots, \|z_B^{(k)}\|_{\infty, \hat{\Pi}_k}$.
\While{$\cbr[1]{j\in \hat{\Pi}_k: |S^{\Delta}_{W,j}|\leq \hat{c}_{1-\alpha, \hat{\Pi}_k}(\tilde{\Sigma}^{\Delta}_{0})} \subsetneqq \hat{\Pi}_k$}  
\State $k \gets k+1$
\State $\hat{\Pi}_k \gets \cbr[1]{j\in \hat{\Pi}_{k-1}: |S^{\Delta}_{W,j}|\leq \hat{c}_{1-\alpha, \hat{\Pi}_{k-1}}(\tilde{\Sigma}^{\Delta}_{0})}$
\If{$\hat{\Pi}_k = \emptyset$} 
\State \textbf{break}
\EndIf
\State repeat steps 7 \& 8 
\EndWhile
\EndIf
\State \textbf{Return:} $\hat{\Pi}^{\Delta}_{B,e} \gets \hat{\Pi}_k$.
\end{algorithmic}
\end{algorithm}

As our last theoretical result, we analyze the FWER for Algorithm~\ref{algo:bootts2}. [In the formulation of Theorem~\ref{thm:mainthBts2} it is understood, cf.~also the definition of~$\tilde{\Sigma}^{\Delta}$ in~\eqref{eqn:covesttspo}, that the pooled covariance estimator is used only in case it is known and correct that~$\Sigma^X = \Sigma^Y$.]
\begin{theorem}\label{thm:mainthBts2}
Let Assumptions~\ref{ass:setting},~\ref{ass:settingY}, and~\ref{ass:indep} be satisfied, fix~$\alpha \in (0, 1)$, and choose a natural number~$B \geq 2$ such that~$Q_{B, 1-\alpha}((1-1/B)^{1/d}) < B$. Then, 
\begin{equation}\label{eqn:fwerBts2}
\P\left(J(\Delta) \not \subseteq \hat{\Pi}_{B,e}^{\Delta}  \right) \leq \alpha + B^{-1} +  C \times \left(\mathfrak{A}^X_{n_X} + \mathfrak{A}^Y_{n_Y}+\mathfrak{B}^{\Delta}+ \mathfrak{C}^{\Delta} \right),
\end{equation}
where $C$ is a constant depending only on~$b^V_1,b^V_2,\lambda_{1, V},\lambda_{2, V},\lambda_{1, V}',\lambda_{2, V}'$ and~$m_V$ for~$V \in \{X, Y\}$. The upper bound in~\eqref{eqn:fwerBts2} converges to~$\alpha$ in any asymptotic regime satisfying Assumption~\ref{ass:asyXY} under the additional condition that~$B \to \infty$.
\end{theorem}

In comparison to the upper bounds in Theorems~\ref{thm:mainthGts} and~\ref{thm:mainthBts}, the upper bound in~\eqref{eqn:fwerBts2} also incorporates the term~$B^{-1}$, which, however, can be made as small as one wants by choosing~$B$ appropriately.

\section{Conclusion}

We introduced quantile-winsorized one- and two-sample multiple testing procedures for restrictions on high-dimensional means. As opposed to tests based on sample averages and covariances, our procedures are robust to adversarial contamination. Furthermore, the dimension of the observations can grow exponentially in sample size, despite only requiring slightly more than two moments to exist, which, even \emph{without} any adversarial robustness, is another benefit of basing inference on winsorized estimators. 

From a technical perspective, our results are based on the general multiple testing procedures in \cite{romano2005exact}, together with Gaussian approximation inequalities for the distribution of high-dimensional \emph{quantile-winsorized} means, which we extend (including covariance estimation) from the one-sample case considered in~\cite{kphdapprox} and~\cite{kprobfree} to the two-sample case in this article, and which may be of some independent interest. 

\bibliographystyle{ecta} 
\bibliography{ref}		

\pagebreak

\begin{appendix}
\numberwithin{equation}{section}
	
\section{Auxiliary results}\label{app:aux}
\numberwithin{equation}{section}	
The main technical tool underlying our proofs are Gaussian approximation results for the distribution of winsorized mean vectors. For the one-sample case, such results were recently obtained in \cite{kphdapprox}; see also~\cite{kprobfree} for generalizations. Let us re-state some of those results briefly in the form that they will be used in the present article (in particular, we equip some of the quantities with an index~$X$, which will be useful for our two-sample results, and we lighten notation somewhat at other places, e.g., we do not show the dependence of some quantities on sample size). We denote by~$\mathcal{H}$ the class of (generalized) hyperrectangles in~$\R^d$, where a (generalized) hyperrectangle is defined as
\begin{equation*}
\left\{x\in\R^d:a_j\leq x_j\leq b_j\text{ for all }j=1,\hdots, d\right\},
\end{equation*}
where~$-\infty\leq a_j\leq b_j\leq \infty$ for every~$j = 1, \hdots, d$. In addition to the vector of normalized winsorized means~$S^X_{W}$, which was defined in~\eqref{eqn:SXWdef}, we denote by~$S^X_{W,S}$ the vector of normalized and \emph{centered} winsorized means with elements 
\begin{equation*}
S^X_{W,S,j}=\frac{1}{\sqrt{n_X}\tilde{\sigma}^X_{j}}\sum_{i=1}^{n_X}\sbr[1]{\phi_{\hat\alpha_j^X,\hat\beta_j^X}(\tilde{X}_{i,j})-\mu^X_j},\qquad j=1,\hdots,d;
\end{equation*}
as in \cite{kphdapprox} we leave the quotients undefined if one of the variance estimators equals~$0$. 

We note for later use (recall that~$S^X_{W,j} = 0$ by definition if $\tilde{\sigma}^X_{j} = 0$, whereas~$S^X_{W,S,j}$ is undefined in that case) that
\begin{equation}\label{eqn:coinc}
S^X_{W,S,j} = S^X_{W,j} \quad \text{ if } \quad \tilde{\sigma}^X_{j} \neq 0 \text{ and } j \in J(\mu^X),
\end{equation}
and that, similarly,
\begin{equation*}
S^{\dagger, X}_{W,S,j} = S^{\dagger, X}_{W,j} \quad \text{ if } \quad j \in J(\mu^X).
\end{equation*}

We start with Theorem~\ref{thm:HDGauss}, a Gaussian approximation result for the distribution of~$S^{\dagger, X}_{W}$. This result is only a notational re-formulation of Theorem 2.1 in \cite{kphdapprox}, but otherwise identical. Note that the quantity~$\mathfrak{A}_{n_X}^X$, which appears in the upper bound in~\eqref{eq:Gaussapprox}, was defined in Equation~\eqref{eqn:abcdef}. [Recall that in the present paper it is assumed throughout that~$d\geq 2$ and that sample sizes are greater than~$3$, which is also maintained in \cite{kphdapprox}.] We emphasize again that throughout the present paper, we assume, without stating it explicitly (e.g., as in \cite{kphdapprox}), that~$\eps_X$ and~$\eps_X'$ are the quantities that were defined in~\eqref{eq:epsfamx}.
\begin{theorem}\label{thm:HDGauss}
Let Assumption~\ref{ass:setting} be satisfied. If~$\eps_X\in(0,1/2)$, then, for~$Z\sim \mathsf{N}_d(0, \Sigma^X)$ and~$C$ a constant depending only on~$b_1^X,b_2^X,\lambda_{1,X},\lambda_{2,X}$, and~$m_X$,
\begin{equation}\label{eq:Gaussapprox}
\rho^X_{W} :=\sup_{H\in\mc{H}}\envert[2]{\P\del[1]{S^{\dagger, X}_{W, S}\in H}-\P\del[1]{Z\in H}} 
\leq 	
C \times  \mathfrak{A}^X_{n_X}.
\end{equation}
In particular,~$\rho^X_{W} \to 0$ in any asymptotic regime satisfying Assumption~\ref{ass:asyX}.
\end{theorem}

We move on to an analogous result for~$S^{X}_{W, S}$, which is a notational re-formulation of Theorem 4.1 in \cite{kphdapprox}. Note that the quantities~$\mathfrak{A}_{n_X}^X$ and~$\mathfrak{B}_{n_X}^X$, which appear in the upper bound in~\eqref{eq:HDGaussstudentized}, were defined in Equation~\eqref{eqn:abcdef}.
\begin{theorem}\label{thm:HDGauss_studentized}
Let Assumption~\ref{ass:setting} be satisfied.  If~$\eps_X,\eps_X'\in(0,1/2)$, then, for~$Z'\sim\mathsf{N}_d(0,\Sigma^X_0)$ and~$C$ a constant depending only on~$b^X_1,b^X_2,\lambda_{1,X},\lambda_{2,X}$, $\lambda_{1,X}',\lambda_{2,X}'$ and~$m_X$,
\begin{equation}\label{eq:HDGaussstudentized}
\rho^X_{W,S}
:=
\sup_{H\in\mc{H}}\envert[2]{\P\del[1]{S^X_{W,S}\in H}-\P\del[1]{Z'\in H}}
\leq
C \times \left(\mathfrak{A}^X_{n_X}+\mathfrak{B}^X_{n_X}\right).
\end{equation}
In particular, the upper bound in~\eqref{eq:HDGaussstudentized}, and hence~$\rho^X_{W,S}$, converges to~$0$ in any asymptotic regime satisfying Assumption~\ref{ass:asyX}.
\end{theorem}
\begin{remark}[Exceptional set where~$\min_{j = 1, \hdots, d} \tilde{\sigma}^X_{j} = 0$]\label{rem:div0}
In \cite{kphdapprox} the quantity~$S^X_{W,S}$ is (on purpose) \emph{not} defined on the event where at least one diagonal element~$\tilde{\sigma}^X_{j}$ of~$\tilde{\Sigma}^X$ equals zero. The event~$\{S^X_{W,S} \in H\}$ appearing in~\eqref{eq:HDGaussstudentized} of Theorem~\ref{thm:HDGauss_studentized} therefore implicitly requires~$S^X_{W,S}$ to be well defined, and hence has to be interpreted as~$\{S^X_{W,S} \in H, \min_{j = 1, \hdots, d} \tilde{\sigma}^X_{j} > 0\}$. As a point of interest, the probability of the event~$\min_{j = 1, \hdots, d} \tilde{\sigma}^X_{j} = 0$ is controlled in Proposition~\ref{prop:BSAlwaysvalid} below, and is negligible in instances where the condition in~\eqref{eq:epscond'} is satisfied, cf.~also Remark~4.1 in \cite{kphdapprox}.
\end{remark}

\section{Proofs for results in Section~\ref{sec:onesample}}\label{app:proofsos}

\emph{For notational convenience, in the proofs in this section, we drop the index~$X$ whenever no confusion may arise; e.g., we write~$\eps$ instead of~$\eps_X$,~$n$ instead of~$n_X$,~$\mu$ instead of~$\mu^X$, and~$S_{W, S}$ instead of~$S_{W, S}^X$.} We start with the following auxiliary result which is a slightly improved version of Lemma C.2 of \cite{kprobfree}. Recall the definition of~$\mathfrak{d}_n$ (we also drop the index~$X$ here) from Equation~\eqref{eqn:abcdef}.
\begin{proposition}\label{prop:BSAlwaysvalid}
Let Assumption~\ref{ass:setting} be satisfied. There exist constants~$C$ and~$C^*$, depending only on~$b_1, b_2, \lambda_1', \lambda_2'$, and~$m$, such that under the condition
\begin{equation}\label{eq:epscond'}
\mathfrak{d}_n <\frac{b_1^2}{2C} \quad \text{ and } \quad	2\eps' +\frac{\log(d^2n)}{n}+\sqrt{\del[2]{\frac{\log(d^2n)}{n}}^2+4\frac{\log(d^2n)}{n}\eps'}<1,
\end{equation}
for~$\mathfrak{s}_n:=C^*\log(d)\mathfrak{d}^{1/2}_n$ with probability at least~$1-24/n$: $\min_{j=1,\hdots,d} \tilde{\sigma}^X_j > 0$ and
\begin{equation}\label{eq:critvalcontrol}
c_{(\beta-\mathfrak{s}_n)\vee 0, A}(\Sigma_{0})\leq c_{\beta,  A}(\tilde{\Sigma}_{0}) \leq c_{(\beta+\mathfrak{s}_n) \wedge 1, A}(\Sigma_{0}) \text{ for every } \emptyset \neq A \subseteq [d] \text{ and } \beta \in [0, 1].
\end{equation}
\end{proposition}

\begin{proof}
We argue as in the proof of Lemma~C.2 in Appendix~C of \cite{kprobfree}, cf.~also Theorem~3.1 in \cite{kphdapprox}, to obtain (under the maintained assumptions, with a suitable choice of the constants~$C$ and~$C^*$ as indicated in the present Proposition's statement, and noting that the argument given in that proof does not depend on how one defines~$\tilde{\Sigma}_{0}$ on the set where one of the diagonal entries of~$\tilde{\Sigma}$ vanishes) that for~$Z \sim \mathsf{N}_d(0, I_d)$ (independent of~$\tilde{X}_1,\hdots,\tilde{X}_{n_X}$) and on an event~$G$, say, with probability greater than or equal to~$1-24/n$, it holds that
\begin{equation*}
\sup_{H\in\mc{H}}\envert[2]{\P\del[1]{\Sigma_{0}^{1/2}Z\in H}-\P\del[1]{\tilde{\Sigma}_{0}^{1/2}Z\in H\mid\tilde{X}_1,\hdots,\tilde{X}_n}}\leq \mathfrak{s}_n \quad \text{ and } \quad \min_{j = 1, \hdots, d} \tilde{\sigma}_{j} >0.
\end{equation*} 
On~$G$, for every~$\emptyset \neq A \subseteq [d]$ and~$\beta \in [0, 1]$, it holds that
\begin{align*}
&\P\del[2]{\max_{j \in A} |(\tilde{\Sigma}_{0}^{1/2}Z)_j|\leq c_{(\beta+\mathfrak{s}_n) \wedge 1, A}(\Sigma_{0})\mid \tilde{X}_1,\hdots, \tilde{X}_n}  \\ 
& \geq  \beta \geq \P\del[2]{\max_{j \in A} |(\tilde{\Sigma}_{0}^{1/2}Z)_j|\leq c_{(\beta-\mathfrak{s}_n)\vee 0, S}(\Sigma_{0})\mid \tilde{X}_1,\hdots, \tilde{X}_n},
\end{align*}
and hence
$c_{(\beta-\mathfrak{s}_n)\vee 0, A}(\Sigma_{0}) \leq c_{\beta,  A}(\tilde{\Sigma}_{0}) \leq c_{(\beta+\mathfrak{s}_n) \wedge 1, A}(\Sigma_{0})$
by definition of~$c_{\beta,  A}(\tilde{\Sigma}_{0})$,~cf.~\eqref{eqn:bootqdef}.
\end{proof}

\begin{proof}[Proof of Theorem~\ref{thm:mainthG}:]
In case~$\max(\eps, \eps') \geq 1/2$, by definition~$\hat{\Pi}_{G} = [d] \supseteq J(\mu)$, so that in this case~$\P(J(\mu) \not \subseteq \hat{\Pi}_G ) = 0$, and the inequality in~\eqref{eqn:fwerG} trivially holds. Hence, we assume from now on that~$\max(\eps, \eps') < 1/2$. We also assume that~$J(\mu) \neq \emptyset$, since otherwise $\P(J(\mu) \not \subseteq \hat{\Pi}_G )  = 0$, and, again, the inequality in~\eqref{eqn:fwerG} trivially holds. Under these assumptions, the relation~$J(\mu) \not \subseteq \hat{\Pi}_G $ implies the existence of an index~$\hat{j} = j(\tilde{X}_1,\hdots, \tilde{X}_n)\in J(\mu)$ and a step~$\hat{k}=k(\tilde{X}_1,\hdots, \tilde{X}_n)$, say, in the while loop of Algorithm~\ref{algo:gauss}, at which
\begin{equation*}
|S_{W,\hat{j}}|>c_{1-\alpha,\hat{\Pi}_{\hat{k}}} \quad \text{ and } \quad J(\mu)\subseteq \hat{\Pi}_{\hat{k}}.
\end{equation*}
The monotonicity relationship~\eqref{eq:MonCVG} therefore implies
\begin{equation*}
\max_{j\in J(\mu)}|S_{W,j}|
\geq
|S_{W,\hat{j}}|
> 
c_{1-\alpha,\hat{\Pi}_{\hat{k}}}
\geq 
c_{1-\alpha,J(\mu)}.
\end{equation*}
We can hence conclude that
\begin{equation*}
\P\left(J(\mu) \not \subseteq \hat{\Pi}_G \right)
\leq
\P \left(\max_{j\in J(\mu)}|S_{W,j}|>c_{1-\alpha,J(\mu)}\right) = 1 - \P \left( S_{W} \in H(c_{1-\alpha,J(\mu)}, \mu)\right),
\end{equation*}
where we set
\begin{equation}\label{eqn:genhypproof}
H(c, \mu) := \left\{y \in \R^d: a_j(c, \mu) \leq y_j \leq b_j(c, \mu), ~j = 1, \hdots, d  \right\},
\end{equation} and where
\begin{equation}\label{eqn:genhypproofab}
a_j(c, \mu) :=
\begin{cases}
- \infty & \text{ if } j \notin J(\mu), \\
- c & \text{ if }  j \in J(\mu) ,
\end{cases}
\quad \text{ and } \quad 
b_j(c, \mu) :=
\begin{cases}
\infty & \text{ if } j \notin J(\mu), \\
c & \text{ if } j \in J(\mu) .
\end{cases}
\end{equation}
Obviously, it holds that (recall the statement in~\eqref{eqn:coinc})
\begin{align*}
\P \left( S_{W} \in H(c_{1-\alpha,J(\mu)}, \mu)\right) & \geq \P \left( S_{W} \in H(c_{1-\alpha,J(\mu)}, \mu), \min_{j = 1, \hdots, d} \tilde{\sigma}_{j} > 0\right) \\ 
& = \P \left( S_{W, S} \in H(c_{1-\alpha,J(\mu)}, \mu), \min_{j = 1, \hdots, d} \tilde{\sigma}_{j} > 0\right)
\end{align*}
Thus, it follows from Theorem~\ref{thm:HDGauss_studentized}, cf.~also Remark~\ref{rem:div0}, that for~$Z\sim\mathsf{N}_d(0,I_d)$
\begin{align*}
\P \left( S_{W} \in H(c_{1-\alpha,J(\mu)}, \mu)\right) & \geq \P(\Sigma_0^{1/2}Z \in H(c_{1-\alpha,J(\mu)}, \mu)) - C \times (\mathfrak{A}_n+\mathfrak{B}_n) \\
& = \P(\max_{j \in J(\mu)} |(\Sigma_0^{1/2} Z)_j|  \leq c_{1-\alpha,J(\mu)}) - C \times  (\mathfrak{A}_n+\mathfrak{B}_n) \\
& \geq (1-\alpha) - C \times  (\mathfrak{A}_n+\mathfrak{B}_n),
\end{align*}
where~$C$ is a constant depending only on~$b_1,b_2,\lambda_1,\lambda_2,\lambda_1',\lambda_2'$ and~$m$, and where the second inequality follows from~\eqref{eqn:ksiquant}. The remaining (convergence) statement follows from the last statement in Theorem~\ref{thm:HDGauss_studentized}.
\end{proof}

\begin{proof}[Proof of Theorem~\ref{thm:mainthB}:]
As in the proof of Theorem~\ref{thm:mainthG}, we assume, without loss of generality, that~$\max(\eps, \eps') < 1/2$ and~$J(\mu) \neq \emptyset$. We then argue similarly as in the proof of Theorem~\ref{thm:mainthG}, but now based on the monotonicity relation in~\eqref{eq:MonCVB}, to conclude 
\begin{equation}\label{eqn:redstep}
\P\left(J(\mu) \not \subseteq \hat{\Pi}_B\right)
\leq
\P \left(\max_{j\in J(\mu)}|S_{W,j}|>c_{1-\alpha,  J(\mu)}(\tilde{\Sigma}_{0})\right) = 1 - \P \left( S_{W} \in H(c_{1-\alpha,  J(\mu)}(\tilde{\Sigma}_{0}), \mu)\right),
\end{equation}
cf.~Equations~\eqref{eqn:genhypproof} and~\eqref{eqn:genhypproofab} for the definition of~$H(c_{1-\alpha,  J(\mu)}(\tilde{\Sigma}_{0}), \mu)$, which is data-\emph{dependent} through~$c_{1-\alpha,  J(\mu)}(\tilde{\Sigma}_{0})$. Hence, in contrast to the argument used to establish Theorem~\ref{thm:mainthG} above, we \emph{cannot} directly apply the Gaussian approximation result in Theorem~\ref{thm:HDGauss_studentized}, but we first relate~$c_{1-\alpha,  J(\mu)}(\tilde{\Sigma}_{0})$ to suitably chosen quantiles from~$c_{\cdot,  J(\mu)}(\Sigma_{0})$ via Proposition~\ref{prop:BSAlwaysvalid}. We argue as follows: 

\textbf{Case 1: condition in~\eqref{eq:epscond'} is satisfied.} By Proposition~\ref{prop:BSAlwaysvalid}, with probability at least~$1-24/n$, 
\begin{equation}\label{eqn:critvalrel}
\min_{j=1,\hdots,d} \tilde{\sigma}_j > 0 ~~ \text{ and } ~~
c_{(1-\alpha-\mathfrak{s}_n)\vee 0, J(\mu)}(\Sigma_{0})\leq c_{1-\alpha,  J(\mu)}(\tilde{\Sigma}_{0}) \leq c_{(1-\alpha +\mathfrak{s}_n) \wedge 1, J(\mu)}(\Sigma_{0}),
\end{equation}
and hence~$$\min_{j=1,\hdots,d} \tilde{\sigma}_j > 0 ~~ \text{ and } ~~ H(c_{(1-\alpha-\mathfrak{s}_n)\vee 0, J(\mu)}(\Sigma_{0}), \mu) \subseteq H(c_{1-\alpha,  J(\mu)}(\tilde{\Sigma}_{0}), \mu).$$ Recalling the statement in~\eqref{eqn:coinc}, we see that
\begin{align*}
&\P \left( S_{W} \in H(c_{1-\alpha,  J(\mu)}(\tilde{\Sigma}_{0}), \mu)\right)  \\ 
&\geq  \P \left( S_{W, S} \in H(c_{(1-\alpha-\mathfrak{s}_n)\vee 0, J(\mu)}(\Sigma_{0}), \mu),  \{\eqref{eqn:critvalrel}\}\right), \\
& \geq  \P \left( S_{W, S} \in H(c_{(1-\alpha-\mathfrak{s}_n)\vee 0, J(\mu)}(\Sigma_{0}), \mu), \min_{j = 1, \hdots, d} \tilde{\sigma}_{j}> 0\right) - \frac{24}{n};
\end{align*}
(where we used~$\P(A, B) \geq \P(A) + \P(B) - 1$ with $B = \{\eqref{eqn:critvalrel}\}$ to obtain the last inequality). Thus, it follows from Theorem~\ref{thm:HDGauss_studentized}, cf.~also Remark~\ref{rem:div0}, that for~$Z\sim\mathsf{N}_d(0,I_d)$
\begin{align*}
& \P \left( S_{W} \in H(c_{1-\alpha,  J(\mu)}(\tilde{\Sigma}_{0}), \mu)\right) \\ & \geq \P(\Sigma_0^{1/2}Z \in H(c_{(1-\alpha-\mathfrak{s}_n)\vee 0, J(\mu)}(\Sigma_{0}), \mu)) - \frac{24}{n} - C \times (\mathfrak{A}_n+\mathfrak{B}_n) \\
& = \P(\max_{j \in J(\mu)} |(\Sigma_0^{1/2} Z)_j|  \leq c_{(1-\alpha-\mathfrak{s}_n)\vee 0,J(\mu)}(\Sigma_{0})) - \frac{24}{n} - C \times  (\mathfrak{A}_n+\mathfrak{B}_n) \\
& \geq (1-\alpha) - \mathfrak{s}_n - \frac{24}{n}  - C \times  (\mathfrak{A}_n+\mathfrak{B}_n) \geq (1-\alpha) - C \times  (\mathfrak{A}_n+\mathfrak{B}_n + \mathfrak{C}_n),
\end{align*}
where~$C$ (which was updated in the last inequality to bound the terms~$\mathfrak{s}_n$ and~$24/n$) is a constant depending only on~$b_1,b_2,\lambda_1,\lambda_2,\lambda_1',\lambda_2'$ and~$m$, and where the second inequality follows from~\eqref{eqn:ksiquant}.

For the remaining cases we adapt arguments from the proof of Theorem B.2 in \cite{kprobfree}:

\textbf{Case 2:~$\mathfrak{d}_n \geq b_1^2/(2C)$.}
In this case, the upper bound in Theorem~\ref{thm:mainthB} exceeds~$1$ for every constant exceeding~$2C/b_1^2$, so that updating the constant (if necessary, and without violating the required dependence relations) takes care of this case (cf.~\eqref{eqn:abcdef} for the definition of~$\mathfrak{d}_n$ and the quantities showing up in the upper bound in Theorem~\ref{thm:mainthB}). 

\textbf{Case 3:~$2\eps' +\frac{\log(d^2n)}{n}+\sqrt{\del[2]{\frac{\log(d^2n)}{n}}^2+4\frac{\log(d^2n)}{n}\eps'}\geq1$.} It follows from~$\frac{\log(d^2n)}{n}\leq {\lambda_2'}^{-1}{\eps'}$ that, for a constant~$C_2$ depending only on~$\lambda_2'$, the expression on the left-hand side of the second inequality in~\eqref{eq:epscond'} is bounded from above by
\begin{align*}
2\eps'+{\lambda_2'}^{-1}{\eps'}+\sqrt{{\lambda_2'}^{-2}{\eps'}^2+4{\lambda_2'}^{-1}{\eps'}^2}
\leq
C_2\eps'.	
\end{align*}
It hence follows that~$C_2 \eps' \geq 1$, and hence, by definition of~$\eps'$, cf.~\eqref{eq:epsfamx},
\begin{align*}
\overline{\eta}\geq \frac{1}{2C_2\lambda_1'}\qquad\text{or}\qquad \frac{\log(dn)}{n}
=
\frac{\log([dn]^2)/2}{n}
\geq
\frac{\log(d^2n)}{2n}
\geq
\frac{1}{4C_2\lambda_2'}.
\end{align*} 
As in Case 2 above, the upper bound in Theorem~\ref{thm:mainthB} remains valid also in Case 3 by suitably adjusting~$C$ (if necessary, and without violating the required dependence relations). 

The remaining (convergence) statement follows from the last statement in Theorem~\ref{thm:HDGauss_studentized} together with Lemma B.4 of \cite{kphdapprox}.
\end{proof}

\begin{proof}[Proof of Lemma~\ref{lem:auxB}:]
Since~$Q_{B, \beta}(\gamma) $ is always a natural number, it holds that
\begin{equation*}
\sum_{i = 1}^B \mathds{1}\{z_i < Q_F(\beta)\} \leq Q_{B, \beta}(\gamma)  \quad  \Rightarrow \quad z^*_{ Q_{B, \beta}(\gamma)  + 1} \geq Q_F(\beta).
\end{equation*}
The sum to the left is a binomially distributed random variable with sample size~$B$ and success probability~$\beta$ (recall that~$F$ is assumed to be continuous). Hence, the probability of the event to the left is at least~$\gamma$.
\end{proof}

\begin{proof}[Proof of Theorem~\ref{thm:mainthB2}:]
We adapt the argument used in the proof of Theorem~\ref{thm:mainthB} and combine it with Lemma~\ref{lem:auxB} (note that we only need to consider Case 1 in the proof of Theorem~\ref{thm:mainthB}, i.e., when the condition in~\eqref{eq:epscond'} is satisfied, due to the structure of the bound claimed in Theorem~\ref{thm:mainthB2}; cf.~the argument used to establish the other cases in the proof of Theorem~\ref{thm:mainthB} for details). We assume, without loss of generality, that~$\max(\eps, \eps') < 1/2$ and~$J(\mu) \neq \emptyset$. We next note that the relation~$J(\mu) \not \subseteq \hat{\Pi}_{B,e}$ implies the existence of an index~$\hat{j} = j(\tilde{X}_1,\hdots, \tilde{X}_n)\in J(\mu)$ and a step~$\hat{k}=k(\tilde{X}_1,\hdots, \tilde{X}_n)$, say, in the while loop of Algorithm~\ref{algo:boot2}, at which
\begin{equation*}
|S_{W,\hat{j}}|> \hat{c}_{1-\alpha, \hat{\Pi}_{\hat{k}}}(\tilde{\Sigma}^{X}_{0}) \quad \text{ and } \quad J(\mu)\subseteq \hat{\Pi}_{\hat{k}}.
\end{equation*}
Recall that by definition~$\hat{c}_{1-\alpha, \hat{\Pi}_{\hat{k}}}(\tilde{\Sigma}^{X}_{0})$ is the~$m(B, \alpha, d)$th order statistic from the (conditionally) i.i.d.~sample~$$\|z_1^{(\hat{k})}\|_{\infty, \hat{\Pi}_{\hat{k}}}, \hdots, \|z_B^{(\hat{k})}\|_{\infty, \hat{\Pi}_{\hat{k}}},$$ which obviously is not smaller than the~$m(B, \alpha, d)$th order statistic~$o_{J(\mu)}^{(\hat{k})}$, say, of the (conditionally) i.i.d.~sample~$$\|z_1^{(\hat{k})}\|_{\infty, J(\mu)}, \hdots, \|z_B^{(\hat{k})}\|_{\infty, J(\mu)}.$$ We hence obtain
\begin{equation*}
\max_{j\in J(\mu)}|S_{W,j}|
\geq
|S_{W,\hat{j}}|
> 
\hat{c}_{1-\alpha, \hat{\Pi}_{\hat{k}}}(\tilde{\Sigma}^{X}_{0})
\geq 
o_{J(\mu)}^{(\hat{k})},
\end{equation*}
and conclude (with the argument used in Case 1 in the proof of Theorem~\ref{thm:mainthB} applied to the first summand in the first subsequent upper bound) that for a suitable constant, with the required dependence properties,
\begin{equation}
\begin{aligned}\label{eqn:redstep}
\P\left(J(\mu) \not \subseteq \hat{\Pi}_{B,e}\right)
& \leq
\P \left(\max_{j\in J(\mu)}|S_{W,j}|>c_{1-\alpha,  J(\mu)}(\tilde{\Sigma}_{0})\right) + \P\left(c_{1-\alpha,  J(\mu)}(\tilde{\Sigma}_{0}) > o_{J(\mu)}^{(\hat{k})}\right) \\
& \leq \alpha + C \times \left( \mathfrak{A}^X_{n_X} + \mathfrak{B}^X_{n_X} + \mathfrak{C}^X_{n_X} \right) + \P\left(c_{1-\alpha,  J(\mu)}(\tilde{\Sigma}_{0}) > o_{J(\mu)}^{(\hat{k})}\right).
\end{aligned}
\end{equation}
It remains to verify that
\begin{equation}
\P\left(c_{1-\alpha,  J(\mu)}(\tilde{\Sigma}_{0}) > o_{J(\mu)}^{(\hat{k})}\right) \leq B^{-1}.
\end{equation}
Clearly, there are at most~$d$ steps in Algorithm~\ref{algo:boot2}, and we can implicitly imagine~$j = 1, \hdots, d$ independent underlying samples of i.i.d.~random vectors~$z_1^{(j)}, \hdots, z_B^{(j)} \sim \mathsf{N}_{d}(0, \tilde{\Sigma}^{X}_{0})$ conditionally on~$\tilde{X}_1, \hdots, \tilde{X}_n$ (where not all of those samples may be ``used'' in every instance of calling the algorithm, depending on whether the algorithm stops in fewer than~$d$ steps or not). Let us denote by~$o_{J(\mu)}^{(k)}$ the $m(B, \alpha, d)$th order statistic from~$$\|z_1^{(k)}\|_{\infty, J(\mu)}, \hdots, \|z_B^{(k)}\|_{\infty, J(\mu)}.$$ Clearly, by the independence and distributional assumptions imposed on the random variables generated in Algorithm~\ref{algo:boot2},
\begin{align*}
\P\left(c_{1-\alpha,  J(\mu)}(\tilde{\Sigma}_{0}) > o_{J(\mu)}^{(\hat{k})} \mid \tilde{X}_1, \hdots, \tilde{X}_n \right) &\leq \P\left(c_{1-\alpha,  J(\mu)}(\tilde{\Sigma}_{0}) > \min_{j = 1}^d o_{J(\mu)}^{(j)} \mid \tilde{X}_1, \hdots, \tilde{X}_n  \right) \\
&= 1 - \P\left(c_{1-\alpha,  J(\mu)}(\tilde{\Sigma}_{0}) \leq \min_{j = 1}^d o_{J(\mu)}^{(j)} \mid \tilde{X}_1, \hdots, \tilde{X}_n  \right) \\
&= 1 - \P^d\left(c_{1-\alpha,  J(\mu)}(\tilde{\Sigma}_{0}) \leq o_{J(\mu)}^{(1)} \mid \tilde{X}_1, \hdots, \tilde{X}_n  \right).
\end{align*}
But Lemma~\ref{lem:auxB} applied (conditionally on~$\tilde{X}_1, \hdots, \tilde{X}_n$) and with~$z_i := \|z_i^{(1)}\|_{\infty, J(\mu)}$ for~$i = 1, \hdots, B$,~$\beta = 1-\alpha$ and~$\gamma = (1-1/B)^{1/d}$, noting that then~$c_{1-\alpha,  J(\mu)}(\tilde{\Sigma}_{0}) = Q_F(1-\alpha)$ (in the notation of Lemma~\ref{lem:auxB}), delivers 
\begin{align*}
(1-1/B)^{1/d} = \gamma &\leq 
\P\left(z^*_{Q_{B, \beta}(\gamma) + 1} \geq Q_F(\beta) \mid \tilde{X}_1, \hdots, \tilde{X}_n \right) \\ &=  
\P\left(o_{J(\mu)}^{(1)} \geq c_{1-\alpha,  J(\mu)}(\tilde{\Sigma}_{0}) \mid \tilde{X}_1, \hdots, \tilde{X}_n \right).
\end{align*}
\end{proof}

\section{Proofs for results in Section~\ref{sec:ts}}\label{app:proofsts}

\begin{proof}[Proof of Proposition~\ref{prop:covestimGRAM2s}:]
Theorem 3.1 in \cite{kphdapprox} shows that for every~$V \in \{X, Y\}$, there exists a constant~$C_V$, say, depending only on~$b_2^V,\lambda_{1,V}',\lambda_{2,V}'$, and~$m_V$, such that 
\begin{equation*}
\P\del[4]{\max_{1\leq j,k\leq d}\envert[1]{\tilde{\Sigma}^{V}_{j,k}-\Sigma^{V}_{j,k}}> C_V  \mathfrak{d}_{n_V}^V }
\leq \frac{24}{n_V}.
\end{equation*}
Recall that~$\Sigma^{\Delta} = w_Y^2 \Sigma^X + w_X^2 \Sigma^Y$, as well as
\begin{equation*}
\tilde{\Sigma}^{\Delta} := \begin{cases} w^2_Y \tilde{\Sigma}^X + w^2_X\tilde{\Sigma}^Y, & \text{ in case it is unknown whether or not } \Sigma^X \neq \Sigma^Y, \\
	w_X^2 \tilde{\Sigma}^X + w_Y^2 \tilde{\Sigma}^Y, & \text{ in case it is known that } \Sigma^X = \Sigma^Y,
\end{cases}
\end{equation*}
where we refer to the second case as the pooled case in this proof.

\emph{In the non-pooled case,} obviously,
\begin{equation*}
\max_{1\leq j,k\leq d}\envert[1]{\tilde{\Sigma}^{\Delta}_{j,k}-\Sigma^{\Delta}_{j,k}} \leq w_Y^2 \max_{1\leq j,k\leq d}\envert[1]{\tilde{\Sigma}^{X}_{j,k}-\Sigma^{X}_{j,k}} + w_X^2 \max_{1\leq j,k\leq d}\envert[1]{\tilde{\Sigma}^{Y}_{j,k}-\Sigma^{Y}_{j,k}}.
\end{equation*}
\emph{In the pooled case,}  we can write~$\Sigma^{\Delta} = w_X^2 \Sigma^X + w_Y^2 \Sigma^Y$ and hence
\begin{equation*}
\max_{1\leq j,k\leq d}\envert[1]{\tilde{\Sigma}^{\Delta}_{j,k}-\Sigma^{\Delta}_{j,k}} \leq w_X^2 \max_{1\leq j,k\leq d}\envert[1]{\tilde{\Sigma}^{X}_{j,k}-\Sigma^{X}_{j,k}} + w_Y^2 \max_{1\leq j,k\leq d}\envert[1]{\tilde{\Sigma}^{Y}_{j,k}-\Sigma^{Y}_{j,k}}.
\end{equation*}

In both cases, it follows (together with the initial observation in this proof) that the probability in~\eqref{eq:covestimGramts}, with~$C := \max(C_X, C_Y)$, is bounded from above by
\begin{equation*}
\P\del[4]{\max_{1\leq j,k\leq d}\envert[1]{\tilde{\Sigma}^{X}_{j,k}-\Sigma^{X}_{j,k}}> C  \mathfrak{d}_{n_X}^X } + \P\del[4]{\max_{1\leq j,k\leq d}\envert[1]{\tilde{\Sigma}^{Y}_{j,k}-\Sigma^{Y}_{j,k}}> C  \mathfrak{d}_{n_Y}^Y } \leq 24 \times  \frac{n_Y + n_X}{n_Xn_Y}.
\end{equation*}
\end{proof}

\begin{proof}[Proof of Lemma~\ref{lem:coincts}:]
Equation~\eqref{eqn:equiv} shows that for~$j \in J(\Delta)$, i.e., if~$\Delta_j = 0$, we have~$S^{\dagger, \Delta}_{W, j} = S^{\dagger, \Delta}_{W, S, j}$ Therefore, if also~$\tilde{\sigma}^{{\Delta}}_{j} \neq 0$, then~\eqref{eqn:coincts} follows from~\eqref{eqn:sdagdeltN}.
\end{proof}

\begin{proof}[Proof of Theorem~\ref{thm:HDGaussts}:]
Fix~$H \in \mathcal{H}$, let~$Z_1 \sim \mathsf{N}_d(0, \Sigma^X)$ and~$Z_2 \sim \mathsf{N}_d(0,  \Sigma^Y)$ be independent, and let $(Z_1, Z_2)$ be independent of~$(\tilde{X}_1, \hdots, \tilde{X}_{n_X}, \tilde{Y}_1, \hdots, \tilde{Y}_{n_Y})$. Clearly,~$Z := w_Y Z_1 - w_X Z_2 \sim \mathsf{N}_d(0, \Sigma^{\Delta})$. Recall from Assumption~\ref{ass:indep} that the samples~$(\tilde{X}_1, \hdots, \tilde{X}_{n_X})$ and~$(\tilde{Y}_1, \hdots, \tilde{Y}_{n_Y})$ are independent. Therefore, the Markov kernel~$(B, s) \mapsto \P(w_Y S^{\dagger, X}_{W, S} \in B + s)$ is a regular conditional distribution of~$S^{\dagger, \Delta}_{W, S} = w_Y S^{\dagger, X}_{W, S} - w_X S^{\dagger, Y}_{W, S}$ given~$w_X S^{\dagger, Y}_{W, S}$ (cf.~Equations~\eqref{eqn:sdagdelt} and~\eqref{eqn:winsmeandefc}). Likewise~$(B, s) \mapsto \P(w_Y Z_1 \in B + s)$ is a regular conditional distribution of~$w_Y Z_1 - w_X S^{\dagger, Y}_{W, S}$ given~$w_X S^{\dagger, Y}_{W, S}$. Denoting the distribution of~$w_X S^{\dagger, Y}_{W, S}$ by~$\nu$, say, it holds that
\begin{align*}
& \P(S^{\dagger, \Delta}_{W, S} \in H) - \P(Z \in H) \\ 
& =   \left[\int \left( \P(w_Y S^{\dagger, X}_{W, S} \in H + s) - \P(w_Y Z_1 \in H + s) \right) d\nu(s) \right] \\ & \hspace{2cm} + \left[\P(w_Y Z_1 - w_X S^{\dagger, Y}_{W, S} \in H) - \P(w_Y Z_1 - w_X Z_2 \in H)\right] =: [I] + [II]. 
\end{align*}
Clearly,~$\left| [I] \right| \leq \rho_W^X$, cf.~the notation introduced in Theorem~\ref{thm:HDGauss}. Furthermore, the Markov kernels~$$(B, z) \mapsto \P(- w_X S^{\dagger, Y}_{W, S} \in B - z) \quad \text{ and } \quad (B, z) \mapsto \P(-w_X Z_2 \in B - z)$$ constitute regular conditional distributions of~$w_Y Z_1-w_X S_{W, S}^{\dagger, Y}$ and~$w_Y Z_1-w_X Z_2$, respectively, given~$w_Y Z_1$, so that, writing~$\nu^*$ for the distribution of~$w_Y Z_1$, it holds that
s
\begin{align*}
\left| [II] \right| \leq \int \left|\P(- w_X S^{\dagger, Y}_{W, S} \in H - z) - \P(- w_X Z_2 \in H - z) \right|d\nu^*(z) \leq \rho_W^Y.
\end{align*}
Hence,
\begin{equation}\label{eqn:boundtsG}
\left| \P(S^{\dagger, \Delta}_{W, S} \in H) - \P(Z \in H) \right| \leq \left| [I] \right| + \left| [II] \right| \leq \sum_{V \in \{X, Y\}} \rho_W^V,
\end{equation}
and we can therefore apply Theorem~\ref{thm:HDGauss} to the sample of~``$Y$'' \emph{and} the sample of~``$X$'' to conclude [note that the convergence statement follows from the last part of Theorem~\ref{thm:HDGauss}].
\end{proof}

\begin{remark}\label{rem:carryoverGts}
The inequality in~\eqref{eqn:boundtsG} isolates the argument used to carry over Gaussian approximations from the one-sample to the two-sample situation. We note here that the argument establishing the upper bound in~\eqref{eqn:boundtsG} in the proof of Theorem~\ref{thm:HDGaussts} only made use of the (independence) Assumption~\ref{ass:indep}, but not of Assumptions~\ref{ass:setting} or~\ref{ass:settingY}. 
\end{remark}

\begin{proof}[Proof of Theorem~\ref{thm:HDGauss_studentizedts}:]
We start with~\eqref{eq:HDGaussstudentizedts}. Write~$\mathfrak{T}_n$ for the quantity on the left-hand side of inequality~\eqref{eqn:epscond}. From~$\frac{\log(d^2n_V)}{n_V}\leq {\lambda_{2,V}'}^{-1}{\eps_V'}$ for every~$V \in \{X, Y\}$, it follows that
\begin{equation*}
\mathfrak{T}_n
\leq
\max_{V \in \{X, Y\}} 2\eps_V'+{\lambda_{2, V}'}^{-1}{\eps_V'}+\sqrt{{\lambda_{2, V}'}^{-2}{\eps_V'}^2+4{\lambda_{2, V}'}^{-1}{\eps_V'}^2}
\leq
K \max_{V \in \{X, Y\}} \eps_V',	
\end{equation*}
for a constant~$K$ depending only on~$\lambda_{2, X}'$ and~$\lambda_{2, Y}'$. Consider, for~$\min(b^X_1, b^Y_1) =: b_1^{\Delta}$, the condition
\begin{equation}\label{eq:CaseNormalize}
K \max_{V \in \{X, Y\}} \eps_V' <1 \qquad\text{and}\qquad C_1  \mathfrak{d}^{\Delta}  \leq \frac{(b_1^{\Delta})^2}{2},
\end{equation}
where~$C_1$ is the constant~``$C$'' from Proposition~\ref{prop:covestimGRAM2s} (and~$\mathfrak{d}^{\Delta}$ was defined in~\eqref{eqn:ddefts}), depending only on~$b_2^V,\lambda_{1,V}',\lambda_{2,V}'$, and~$m_V$ for~$V \in \{X, Y\}$. Note that if one of the conditions in~\eqref{eq:CaseNormalize} would be violated, we could easily conclude~\eqref{eq:HDGaussstudentizedts} by a suitable choice of the constant (by an argument as in the proof of Theorem~\ref{thm:mainthB}, Cases~2 and~3). Hence, we shall assume that~\eqref{eq:CaseNormalize} holds. Note that~$K \max_{V \in \{X, Y\}} \eps_V' <1$ implies~\eqref{eqn:epscond} by the penultimate display, so that the inequality in~\eqref{eq:covestimGramts} of Proposition~\ref{prop:covestimGRAM2s} holds. 

Denoting~$D_{\Delta} := \mathrm{diag}(\sigma_{2,1}^{\Delta}, \hdots, \sigma_{2,d}^{\Delta})$ and~$\tilde{D}_{\Delta} := \mathrm{diag}(\tilde \sigma_{1}^{\Delta}, \hdots, \tilde \sigma_{d}^{\Delta})$, we obtain from~\eqref{eqn:sdagdeltN} that (grant none of the variance estimators vanishes)~$$S^{\Delta}_{W,S} = \tilde D_{\Delta}^{-1} S^{\dagger, \Delta}_{W, S}.$$ Define~$	T_{\Delta} := D_{\Delta}^{-1}  S^{\dagger, \Delta}_{W, S}$ and observe that (grant none of the variance estimators vanishes)
\begin{equation}
A^{\Delta}
:=
\enVert[1]{S^{\Delta}_{W,S}-T_{\Delta}}_\infty
\leq 
\max_{j=1,\hdots,d}\envert[3]{\frac{1}{\tilde{\sigma}^{\Delta}_{j}}-\frac{1}{\sigma_{j}^{\Delta}}}
\cdot \| S^{\dagger, \Delta}_{W,S} \|_{\infty}\label{eq:An1}.
\end{equation}
By~\eqref{eq:covestimGramts} of Proposition~\ref{prop:covestimGRAM2s}, using~$\min_{j=1,\hdots,d}\sigma^{\Delta}_{2,j}\geq \min(b^X_1, b^Y_1) = b_1^{\Delta}$ and~\eqref{eq:CaseNormalize}, there exists a constant~$C_2$ depending only on $b_1^V,b_2^V,\lambda_{1,V}',\lambda_{2,V}'$, and~$m_V$, for~$V \in \{X, Y\}$, such that on a set of probability at least~$1-24 \times  \frac{n_Y + n_X}{n_Xn_Y}$ it holds that
\begin{equation}
\min_{j=1,\hdots,d} \tilde \sigma^{\Delta}_{j} \geq b^{\Delta}/\sqrt{2} > 0 ~ \text{ and } ~ \max_{j=1,\hdots,d}\envert[3]{\frac{1}{\tilde{\sigma}^{\Delta}_{j}}-\frac{1}{\sigma^{\Delta}_{2,j}}}
=
\max_{j=1,\hdots,d}\frac{|\tilde{\sigma}^{\Delta}_{j}-\sigma^{\Delta}_{2,j}|}{\tilde{\sigma}^{\Delta}_{j}\sigma^{\Delta}_{2,j}}
\leq 
C_2 \mathfrak{d}^{\Delta}.
\end{equation}
Furthermore, as~$\sigma^{\Delta}_{2,j}\leq \max(\sigma^X_{m_X,j}, \sigma^Y_{m_Y,j}) \leq \max(b_2^X, b_2^Y) =: b_2$, the union bound and~$\P\del[0]{|\bm{z}|> t}\leq 2\exp(-t^2/2)$ for~$t\geq 0$ and~$\bm{z} \sim \mathsf{N}_1(0,1)$ yields for~$Z \sim \mathsf{N}_d(0, \Sigma^{\Delta})$ that
\begin{equation*}
\P\del{\max_{j=1,\hdots,d}|Z_j|> b_2\sqrt{2\log(2d\min(n_X, n_Y))}}
\leq
\frac{1}{\min(n_X, n_Y)}.
\end{equation*}
Thus, by Theorem~\ref{thm:HDGaussts}, there exists a constant~$C^{\Delta}$, depending only on $b_1^V,b_2^V,\lambda_{1,V},\lambda_{2,V}$, and~$m_V$, for~$V \in \{X, Y\}$, such that
\begin{align*}
&\P\del[3]{\| S^{\dagger, \Delta}_{W,S} \|_{\infty}> b_2\sqrt{2\log(2d\min(n_X, n_Y))}} \\
&\leq 
C^{\Delta} \times \left(\mathfrak{A}^X_{n_X} + \mathfrak{A}^Y_{n_Y} \right)+\frac{1}{\min(n_X, n_Y)}
\leq
C_*^{\Delta} \times \left(\mathfrak{A}^X_{n_X} + \mathfrak{A}^Y_{n_Y} \right) \label{eq:An3}, 
\end{align*}
where the value of~$C^{\Delta}$ was suitably adjusted to justify the second inequality. Combining these observations, we conclude that there exists a constants~$C_3$ and~$K_1$, depending only on $b_1^V,b_2^V,\lambda_{1,V},\lambda_{2,V},\lambda_{1,V}',\lambda_{2,V}'$, and~$m_V$, for~$V \in \{X, Y\}$, such that
\begin{align*}
\P\del[3]{A^{\Delta}\leq C_3\sqrt{\log(d\min(n_X, n_Y))} \mathfrak{d}^{\Delta} }
&\geq 
1-24 \times  \frac{n_Y + n_X}{n_Xn_Y}-K_1 \left(\mathfrak{A}^X_{n_X} + \mathfrak{A}^Y_{n_Y} \right) \\
&\geq
1-K \left(\mathfrak{A}^X_{n_X} + \mathfrak{A}^Y_{n_Y} \right),
\end{align*}
where~$K$ depends only on~$b_1^V,b_2^V,\lambda_{1,V},\lambda_{2,V},\lambda_{1,V}',\lambda_{2,V}'$, and~$m_V$, for~$V \in \{X, Y\}$

To conclude~\eqref{eq:HDGaussstudentizedts}, we set up for an application of Lemma A.4 in \cite{kprobfree}, and use the following notation: for any non-empty set~$B\subseteq \R^d$ and~$\zeta >0$, let $$B^{\zeta,\infty}:=\cbr[1]{x\in\R^d:\inf_{y\in B}\|x-y\|_\infty\leq \zeta}.$$ Furthermore, $$B^{-\zeta,\infty}:=\cbr[1]{x\in\R^d:\mc{B}_\infty(x,\zeta) \subseteq B} \quad \text{ where } \quad \mc{B}_{\infty}(x,\zeta):=\cbr[1]{y\in\R^d:\|y-x\|_{\infty} \leq \zeta}.$$
Set~$\overline{A}_{\Delta} := C_3\sqrt{\log(d\min(n_X, n_Y))}\mathfrak{d}^{\Delta}$, fix~$H \in \mathcal{H}$, and note that as an immediate consequence of Theorem~\ref{thm:HDGaussts}, we obtain, for~$Z'\sim\mathsf{N}_d(0,\Sigma_0^{\Delta})$ and~$v \in \{-1, 1\}$, that
\begin{align*}
&\envert[2]{\P\del[1]{T_{\Delta}\in H^{v\overline{A}_{\Delta},\infty}}-\P\del[1]{Z'\in H^{v\overline{A}_{\Delta},\infty}}} \\
&=\envert[2]{\P\del[1]{  S^{\dagger, \Delta}_{W, S} \in D_{\Delta} H^{v\overline{A}_{\Delta},\infty}}-\P\del[1]{D_{\Delta} Z'\in D_{\Delta}H ^{v\overline{A}_{\Delta},\infty}}}   \leq \rho^{\Delta}_W \leq C^{\Delta} \times \left(\mathfrak{A}^X_{n_X} + \mathfrak{A}^Y_{n_Y} \right).
\end{align*}
We now use Lemma A.4 in \cite{kprobfree} with ~$H$,~$U = T_{\Delta}$, $V = S^{\Delta}_{W, S}$, ~$\Xi = \Sigma_{0}^{\Delta}$,~$\xi = 1$,~$a = C^{\Delta} \times \left(\mathfrak{A}^X_{n_X} + \mathfrak{A}^Y_{n_Y} \right)$, $b = \overline{A}_{\Delta}$ and~$c = K \left(\mathfrak{A}^X_{n_X} + \mathfrak{A}^Y_{n_Y} \right)$ to conclude 
\begin{equation*}
\left|
\P(S^{\Delta}_{W, S} \in H) - \P(Z' \in H)
\right|
\leq 
(C^{\Delta} + K) \times \left(\mathfrak{A}^X_{n_X} + \mathfrak{A}^Y_{n_Y} \right) + (\sqrt{2 \log(d)} + 4) \overline{A}_{\Delta},
\end{equation*}
which establishes the required bound.

To establish the final convergence statement, recall, in particular that~$\mathfrak{B}^V_{n_V} := \sqrt{\log(d)\log(dn_V)} \times \mathfrak{d}^V_{n_V}$ and
\begin{equation*}
\mathfrak{d}^{\Delta} 
= 
\begin{cases}
w^2_Y\mathfrak{d}_{n_X}^X + w^2_X \mathfrak{d}_{n_Y}^Y, & \text{ if the non-pooled cov.~est. is used}, \\
w^2_X\mathfrak{d}_{n_X}^X + w^2_Y \mathfrak{d}_{n_Y}^Y, & \text{ if~$\Sigma^X = \Sigma^Y$ and the pooled cov.~est. is used,}
\end{cases}
\end{equation*}
which establishes
\begin{equation*}
	\mathfrak{B}^{\Delta} = \sqrt{\log(d)\log(d\min(n_X, n_Y))} \times \mathfrak{d}^{\Delta} \leq \max( \mathfrak{B}^X_{n_X}, \mathfrak{B}^Y_{n_Y}).
\end{equation*}
Now, the convergence statement follows from Assumptions~\ref{ass:setting} and~\ref{ass:settingY}, and the (one-sample) convergence statement in Theorem~\ref{thm:HDGauss_studentized} (applied to the sample of~``$Y$'' \emph{and} the sample of~``$X$''),
\end{proof}

\begin{proposition}\label{prop:BSAlwaysvalidts}
Let Assumptions~\ref{ass:setting} and~\ref{ass:settingY} be satisfied. There exist constants~$C$ and~$C^*$, depending only on~$b^V_1,b^V_2,\lambda_{1,V},\lambda_{2,V},\lambda_{1,V}',\lambda_{2,V}'$ and~$m_V$ for~$V \in \{X, Y\}$, such that under the condition
\begin{equation}\label{eq:epscond'ts}
\mathfrak{d}^{\Delta}  \leq \frac{(b_1^{\Delta})^2}{2C}, \quad \text{ and } \quad	\eqref{eqn:epscond},
\end{equation}
for~$\mathfrak{s}^{\Delta}:=C^*\log(d)\sqrt{\mathfrak{d}^{\Delta}}$ with probability at least~$1-24 \times  \frac{n_Y + n_X}{n_Xn_Y}$: $\min_{j = 1, \hdots, d} \tilde{\sigma}^{\Delta}_{j} > 0$ and
\begin{equation}\label{eq:critvalcontrolts}
c_{(\beta-\mathfrak{s}^{\Delta}_n)\vee 0, A}(\Sigma_{0}^{\Delta})\leq c_{\beta,  A}(\tilde{\Sigma}_{0}^{\Delta}) \leq c_{(\beta+\mathfrak{s}^{\Delta}_n) \wedge 1, A}(\Sigma_{0}^{\Delta}) \text{ for every } \emptyset \neq A \subseteq [d] \text{ and } \beta \in [0, 1]
\end{equation}
\end{proposition}

The proof proceeds similarly as the proof of Proposition~\ref{prop:BSAlwaysvalid}.

\begin{proof}[Proof of Proposition~\ref{prop:BSAlwaysvalidts}:] Argue as in the proof of Lemma C.2 in Appendix C of \cite{kprobfree}, but now based on our Proposition~\ref{prop:covestimGRAM2s} instead of Theorem 3.1 in \cite{kphdapprox}, to obtain that for~$Z \sim \mathsf{N}_d(0, I_d)$ (independent of~ $\tilde{X}_1,\hdots,\tilde{X}_{n_X}$, and  $\tilde{Y}_1,\hdots,\tilde{Y}_{n_Y}$) and on an event~$G$, say, with probability greater than or equal to~$1-24 \times  \frac{n_Y + n_X}{n_Xn_Y}$, it holds that~$\min_{j =1, \hdots, d} \tilde{\sigma}_j^{\Delta} > 0$ and
\begin{align*}
\sup_{H\in\mc{H}}\envert[2]{\P\del[1]{\Sigma_{0}^{\Delta, 1/2}Z\in H}-\P\del[1]{\tilde{\Sigma}_{0}^{\Delta, 1/2}Z\in H\mid\tilde{X}_1,\hdots,\tilde{X}_{n_X}, \tilde{Y}_1,\hdots,\tilde{Y}_{n_Y}}}
\leq \mathfrak{s}^{\Delta}.
\end{align*} 
On~$G$, for every~$\emptyset \neq A \subseteq [d]$ and~$\beta \in [0, 1]$, it holds that
\begin{equation*}
\begin{aligned}
&\P\del[2]{\max_{j \in A} |(\tilde{\Sigma}_{0}^{\Delta, 1/2}Z)_j|\leq c_{(\beta+\mathfrak{s}^{\Delta}_n) \wedge 1, A}(\Sigma_{0}^{\Delta})\mid\tilde{X}_1,\hdots,\tilde{X}_{n_X}, \tilde{Y}_1,\hdots,\tilde{Y}_{n_Y}} \\ & \geq  \beta \geq \P\del[2]{\max_{j \in A} |(\tilde{\Sigma}_{0}^{\Delta, 1/2}Z)_j|\leq c_{(\beta-\mathfrak{s}_n)\vee 0, A}(\Sigma_{0}^{\Delta})\mid\tilde{X}_1,\hdots,\tilde{X}_{n_X}, \tilde{Y}_1,\hdots,\tilde{Y}_{n_Y}},
\end{aligned}
\end{equation*}
from which $c_{(\beta-\mathfrak{s}^{\Delta}_n)\vee 0, A}(\Sigma^{\Delta}_{0}) \leq c_{\beta,  A}(\tilde{\Sigma}^{\Delta}_{0}) \leq c_{(\beta+\mathfrak{s}^{\Delta}_n) \wedge 1, A}(\Sigma_{0}^{\Delta})$ follows by definition of~$c_{\beta,  A}(\tilde{\Sigma}^{\Delta}_{0})$ given prior to Equation~\eqref{eq:MonCVBts}.
\end{proof}

\begin{proof}[Proof of Theorem~\ref{thm:mainthGtsgl}:]
We can combine Lemma~\ref{lem:coincts} (recall that~$J(\Delta) = [d]$ in the present context) with Theorem~\ref{thm:HDGauss_studentizedts} to obtain for~$Z'\sim\mathsf{N}_d(0,\Sigma_0^{\Delta})$ (cf.~also Remark~\ref{rem:div0}; an analogous statement applies in the present context)
\begin{equation}\label{eqn:intermedgh}
\sup_{H\in\mc{H}}\envert[2]{\P\del[1]{S^{\Delta}_{W}\in H}-\P\del[1]{Z'\in H}}
\leq
C_S^{\Delta} \times \left(\mathfrak{A}^X_{n_X} + \mathfrak{A}^Y_{n_Y}+\mathfrak{B}^{\Delta}\right),
\end{equation}
The upper bound in~\eqref{eqn:fwerGtsG} now immediately follows from the Khatri-{\v{S}}id{\'a}k inequality.

To establish~\eqref{eqn:fwerGtsB} we shall only consider the case where condition in~\eqref{eq:epscond'ts} is satisfied. If this condition does not hold, one can easily increase the constant~$C'$ (without changing its claimed dependence properties) so that the upper bound~\eqref{eqn:fwerGtsB} trivially holds [cf.~Cases 2 and 3 in the proof of Theorem~\ref{thm:mainthBts} for details, which we do not repeat here.] Provided~\eqref{eq:epscond'ts} is satisfied, by Proposition~\ref{prop:BSAlwaysvalidts}, with probability at least~$1-24 \times  \frac{n_Y + n_X}{n_Xn_Y}$, $\min_{j = 1, \hdots, d} \tilde{\sigma}^{\Delta}_{j} > 0$ and (for~$\mathfrak{s}^{\Delta}:=C^*\log(d)\sqrt{\mathfrak{d}^{\Delta}}$)
\begin{equation}\label{eqn:critvalreltsgh}
c_{(1-\alpha-\mathfrak{s}^{\Delta})\vee 0, [d]}(\Sigma^{\Delta}_{0})\leq c_{1-\alpha,  [d]}(\tilde{\Sigma}^{\Delta}_{0}) \leq c_{(1-\alpha +\mathfrak{s}^{\Delta}) \wedge 1, [d]}(\Sigma^{\Delta}_{0}).
\end{equation}
Together with~\eqref{eqn:intermedgh} it follows that
$$\P(\|S^{\Delta}_{W}\|_{\infty} > c_{1-\alpha, [d]}(\tilde{\Sigma}^{\Delta}_0) ) \leq \alpha+ \mathfrak{s}^{\Delta} + 24 \times  \frac{n_Y + n_X}{n_Xn_Y} + C_S^{\Delta} \times \left(\mathfrak{A}^X_{n_X} + \mathfrak{A}^Y_{n_Y}+\mathfrak{B}^{\Delta}\right).$$ Increasing the constant (without changing its dependence properties; cf.~also the proof of Theorem~\ref{thm:HDGauss_studentizedts}), we can omit explicitly stating~$24 \times  \frac{n_Y + n_X}{n_Xn_Y}$ in the upper bound just obtained, so that, recalling the definition of~$\mathfrak{C}^{\Delta}$, the statement in~\eqref{eqn:fwerGtsB} follows.

To establish~\eqref{eqn:fwerGtsB2}, we first note that 
\begin{equation*}
\P\left(\|S^{\Delta}_{W}\|_{\infty} > \hat{c}_{1-\alpha, [d]}(\tilde{\Sigma}^{\Delta}_0) \right) \leq \P\left(\|S^{\Delta}_{W}\|_{\infty} > c_{1-\alpha, [d]}(\tilde{\Sigma}^{\Delta}_0) \right) + \P\left(\hat{c}_{1-\alpha, [d]}(\tilde{\Sigma}^{\Delta}_0) < c_{1-\alpha, [d]}(\tilde{\Sigma}^{\Delta}_0)\right).
\end{equation*}
The first summand in the upper bound is bounded via~\eqref{eqn:fwerGtsB}. It hence suffices to recall that~$$\hat{c}_{1-\alpha, [d]}(\tilde{\Sigma}^{\Delta}_0) := \text{the } (Q_{B, 1-\alpha}(1-1/B) + 1)\text{th order statistic from } \quad \|z_1\|_{\infty}, \hdots, \|z_B\|_{\infty},$$ where~$z_1, \hdots, z_B$ are independently drawn from~$\mathsf{N}_d(0, \tilde{\Sigma}^{\Delta}_0)$ (conditionally on the data), so that Lemma~\ref{lem:auxB} (applied conditionally on the data, with~$z_i = \|z_i\|_{\infty}$,~$\beta = 1-\alpha$,~$\gamma = 1-B^{-1}$),  delivers~$$\P\left(\hat{c}_{1-\alpha, [d]}(\tilde{\Sigma}^{\Delta}_0) \geq c_{1-\alpha, [d]}(\tilde{\Sigma}^{\Delta}_0) \mid \tilde{X}_1, \hdots, \tilde{X}_{n_X}, \tilde{Y}_1, \hdots, \tilde{Y}_{n_Y}\right) \geq 1-B^{-1}.$$

The remaining statement is an immediate consequence of the last statement in Theorem~\ref{thm:HDGauss_studentizedts} together with Lemma B.4 of \cite{kphdapprox}.
\end{proof}

\begin{proof}[Proof of Theorem~\ref{thm:mainthGts}:]
We adapt the argument that was used to establish Theorem~\ref{thm:mainthG}. First of all, without loss of generality, we assume that~$\max(\eps_V, \eps_V': V \in \{X, Y\}) < 1/2$ and~$J(\Delta) \neq \emptyset$. Under these assumptions,~$J(\Delta) \not \subseteq \hat{\Pi}^{\Delta}_G $ implies the existence of an index~$\hat{j} = j(\tilde{X}_1,\hdots, \tilde{X}_{n_X}, \tilde{Y}_1,\hdots, \tilde{Y}_{n_Y})\in J(\Delta)$ and a step~$\hat{k}=k(\tilde{X}_1,\hdots, \tilde{X}_{n_X}, \tilde{Y}_1,\hdots, \tilde{Y}_{n_Y})$, say, in the while loop of Algorithm~\ref{algo:gaussts}, at which (using the monotonicity relationship~\eqref{eq:MonCVG})
\begin{equation*}
\max_{j\in J(\Delta)}|S_{W,j}^{\Delta}| \geq	|S_{W,\hat{j}}^{\Delta}|>c_{1-\alpha,\hat{\Pi}_{\hat{k}}} \geq c_{1-\alpha,J(\Delta)} \quad \text{ and } \quad J(\Delta)\subseteq \hat{\Pi}_{\hat{k}}.
\end{equation*}
We can hence conclude (recalling notation introduced in~\eqref{eqn:genhypproof} and~\eqref{eqn:genhypproofab}) that
\begin{equation*}
\P\left(J(\Delta) \not \subseteq \hat{\Pi}_G^{\Delta} \right)
\leq
\P \left(\max_{j\in J(\Delta)}|S_{W,j}^{\Delta}|>c_{1-\alpha,J(\Delta)}\right) = 1 - \P \left( S_{W}^{\Delta} \in H(c_{1-\alpha,J(\Delta)}, \Delta)\right).
\end{equation*}
By Lemma~\ref{lem:coincts}, it holds that~$S_{W, j}^{\Delta} = S_{W, S, j}^{\Delta}$ if~$\tilde{\sigma}^{\Delta}_{j} \neq 0$ and~$j \in J(\Delta)$. Therefore,
\begin{equation*}
\P \left(  S_{W}^{\Delta} \in H(c_{1-\alpha,J(\Delta)}, \Delta) \right) \geq \P \left( S^{\Delta}_{W, S} \in H(c_{1-\alpha,J(\Delta)}, \Delta), \min_{j = 1, \hdots, d} \tilde{\sigma}^{\Delta}_{j} > 0\right)
\end{equation*}
Thus, it follows from Theorem~\ref{thm:HDGauss_studentizedts}, cf.~also Remark~\ref{rem:div0} (an analogous statement applies here), that for~$Z\sim\mathsf{N}_d(0,I_d)$
\begin{align*}
&\P \left( S_{W}^{\Delta} \in H(c_{1-\alpha,J(\Delta)}, \Delta)\right) \\ & \geq \P(\Sigma_0^{\Delta, 1/2}Z \in H(c_{1-\alpha,J(\Delta)}, \Delta)) - C_S^{\Delta} \times \left(\mathfrak{A}^X_{n_X} + \mathfrak{A}^Y_{n_Y}+\mathfrak{B}^{\Delta}\right) \\
& = \P(\max_{j \in J(\Delta)} |(\Sigma_0^{\Delta, 1/2} Z)_j|  \leq c_{1-\alpha,J(\Delta)}) - C_S^{\Delta} \times \left(\mathfrak{A}^X_{n_X} + \mathfrak{A}^Y_{n_Y}+\mathfrak{B}^{\Delta}\right) \\
& \geq (1-\alpha) - C_S^{\Delta} \times \left(\mathfrak{A}^X_{n_X} + \mathfrak{A}^Y_{n_Y}+\mathfrak{B}^{\Delta}\right),
\end{align*}
where the second inequality follows from the Khatri-{\v{S}}id{\'a}k inequality (more specifically, the argument underlying the first inequality in~\eqref{eqn:ksiquant}). The remaining (convergence) statement follows from the last statement in Theorem~\ref{thm:HDGauss_studentizedts}.
\end{proof}
\begin{proof}[Proof of Theorem~\ref{thm:mainthBts}:]
We adapt the argument used to establish Theorem~\ref{thm:mainthB}. As in the proof of Theorem~\ref{thm:mainthGts}, we assume, without loss of generality, that~$\max(\eps_V, \eps_V': V \in \{X, Y\}) < 1/2$ and~$J(\Delta) \neq \emptyset$. We then argue similarly as in the proof of Theorem~\ref{thm:mainthGts}, but now based on the monotonicity relation in~\eqref{eq:MonCVBts}, that (cf.~Equations~\eqref{eqn:genhypproof} and~\eqref{eqn:genhypproofab})
\begin{equation}
\begin{aligned}\label{eqn:redstepts}
	\P\left(J(\Delta) \not \subseteq \hat{\Pi}^{\Delta}_B\right)
	\leq  1 - \P \left( S^{\Delta}_{W} \in H(c_{1-\alpha,J({\Delta})}(\tilde{\Sigma}^{\Delta}), \Delta)\right),
\end{aligned}
\end{equation}
The generalized hyperrectangle~$H(c_{1-\alpha,J({\Delta})}(\tilde{\Sigma}^{\Delta}), \Delta)$ is data-\emph{dependent}. Hence, in contrast to the argument used to establish Theorem~\ref{thm:mainthGts}, we also need to relate the bootstrap critical values to data-independent critical values via Proposition~\ref{prop:BSAlwaysvalidts} as follows: 

\textbf{Case 1: condition in~\eqref{eq:epscond'ts} is satisfied.} By Proposition~\ref{prop:BSAlwaysvalidts}, with probability at least~$1-24 \times  \frac{n_Y + n_X}{n_Xn_Y}$, $\min_{j = 1, \hdots, d} \tilde{\sigma}^{\Delta}_{j} > 0$ and
\begin{equation}\label{eqn:critvalrelts}
c_{(1-\alpha-\mathfrak{s}^{\Delta})\vee 0, J(\Delta)}(\Sigma^{\Delta}_{0})\leq c_{1-\alpha,  J(\Delta)}(\tilde{\Sigma}^{\Delta}_{0}) \leq c_{(1-\alpha +\mathfrak{s}^{\Delta}) \wedge 1, J(\Delta)}(\Sigma^{\Delta}_{0}),
\end{equation}
and hence~$$H(c_{(1-\alpha-\mathfrak{s}_n^{\Delta})\vee 0, J(\Delta)}(\Sigma^{\Delta}_{0}), \Delta) \subseteq H(c_{1-\alpha,  J(\Delta)}(\tilde{\Sigma}^{\Delta}_{0}), \Delta).$$ By Lemma~\ref{lem:coincts}, it holds that~$S_{W, j}^{\Delta} = S_{W, S, j}^{\Delta}$ if~$\tilde{\sigma}^{\Delta}_{j} \neq 0$ and~$j \in J(\Delta)$. Therefore,
\begin{align*}
&\P \left( S^{\Delta}_{W} \in H(c_{1-\alpha,J({\Delta})}(\tilde{\Sigma}^{\Delta}), \Delta)\right) \\ 
&\geq  \P \left( S^{\Delta}_{W, S} \in H(c_{(1-\alpha-\mathfrak{s}_n^{\Delta})\vee 0, J(\Delta)}(\Sigma^{\Delta}_{0}), \Delta), \min_{j = 1, \hdots, d} \tilde{\sigma}^{\Delta}_{j} > 0, \eqref{eqn:critvalrelts}\right), \\
& \geq  \P \left( S^{\Delta}_{W, S} \in H(c_{(1-\alpha-\mathfrak{s}_n^{\Delta})\vee 0, J(\Delta)}(\Sigma^{\Delta}_{0}), \Delta), \min_{j = 1, \hdots, d} \tilde{\sigma}^{\Delta}_{j} > 0\right) - 24 \times  \frac{n_Y + n_X}{n_Xn_Y};
\end{align*}
(where we used~$\P(A, B) = \P(A) + \P(B) - \P(A\cup B) \geq \P(A) + \P(B) - 1$ with $B = \{\eqref{eqn:critvalrelts}\}$ to obtain the last inequality). Thus, it follows from Theorem~\ref{thm:HDGauss_studentizedts}, that for~$Z\sim\mathsf{N}_d(0,I_d)$
\begin{align*}
& \P \left( S^{\Delta}_{W} \in H(c_{1-\alpha,J({\Delta})}(\tilde{\Sigma}^{\Delta}), \Delta)\right) \\ & \geq \P \left( \Sigma_0^{\Delta, 1/2} Z \in H(c_{(1-\alpha-\mathfrak{s}_n^{\Delta})\vee 0, J(\Delta)}(\Sigma^{\Delta}_{0}), \Delta)\right) \\ & \hspace{3.7cm} - 24 \times  \frac{n_Y + n_X}{n_Xn_Y} - C_S^{\Delta} \times \left(\mathfrak{A}^X_{n_X} + \mathfrak{A}^Y_{n_Y}+\mathfrak{B}^{\Delta}\right) \\
& = \P(\max_{j \in J(\Delta)} |(\Sigma_0^{\Delta, 1/2} Z)_j|  \leq c_{(1-\alpha-\mathfrak{s}_n^{\Delta})\vee 0, J(\Delta)}(\Sigma^{\Delta}_{0})) \\ & \hspace{3.7cm} - 24 \times  \frac{n_Y + n_X}{n_Xn_Y} - C_S^{\Delta} \times \left(\mathfrak{A}^X_{n_X} + \mathfrak{A}^Y_{n_Y}+\mathfrak{B}^{\Delta}\right) \\
& \geq (1-\alpha -\mathfrak{s}_n^{\Delta}) - C \times \left(\mathfrak{A}^X_{n_X} + \mathfrak{A}^Y_{n_Y}+\mathfrak{B}^{\Delta}\right),
\end{align*}
where we updated the constant (depending only on~$b^V_1,b^V_2,\lambda_{1, V},\lambda_{2, V},\lambda_{1, V}',\lambda_{2, V}'$ and~$m_V$ for~$V \in \{X, Y\}$) to obtain the last inequality. The required statement hence follows, recalling the definition of~$\mathfrak{C}^{\Delta}$.

\textbf{Case 2:~$\mathfrak{d}^{\Delta} > \frac{(b_1^{\Delta})^2}{2C}$.}
In this case, the upper bound in Theorem~\ref{thm:mainthBts} exceeds~$1$ for every constant exceeding~$\frac{2C}{(b_1^{\Delta})^2}$, so that updating the constant (if necessary) takes care of this case. 

\textbf{Case 3:~$\max_{V \in \{X, Y\}} \left[2\eps'_V +\frac{\log(d^2n_V)}{n_V}+\sqrt{\del[2]{\frac{\log(d^2n_V)}{n_V}}^2+4\frac{\log(d^2n_V)}{n_V}\eps_V'} \right] \geq1$.} It follows from~$\frac{\log(d^2n_V)}{n_V}\leq {\lambda_{2,V}'}^{-1}{\eps_V'}$ that, for a constant~$C_2$ depending only on~$\lambda_{2,V}'$, 
\begin{align*}
&2\eps'_V +\frac{\log(d^2n_V)}{n_V}+\sqrt{\del[2]{\frac{\log(d^2n_V)}{n_V}}^2+4\frac{\log(d^2n_V)}{n_V}\eps_V'} \\ \leq &	2\eps_V'+{\lambda_{2,V}'}^{-1}{\eps_V'}+\sqrt{{\lambda_{2,V}'}^{-2}{\eps_V'}^2+4{\lambda_{2,V}'}^{-1}{\eps_V'}^2}
\leq
C_2\eps_V'.	
\end{align*}
If~$C_2\eps_V'\geq 1$ then 
\begin{align*}
\overline{\eta}_V\geq \frac{1}{2C_2\lambda_{1,V}'}\qquad\text{or}\qquad \frac{\log(dn_V)}{n_V}
=
\frac{\log([dn_V]^2)/2}{n_V}
\geq
\frac{\log(d^2n_V)}{2n_V}
\geq
\frac{1}{4C_2\lambda_{2,V}'}.
\end{align*} 
In either case, the upper bound in~Theorem~\ref{thm:mainthBts} remains valid, also in this case, by suitably adjusting~$C$ (if needed). Hence, we also established the statement in case the condition in~\eqref{eq:epscond'ts} is violated.

The remaining (convergence) statement follows from the last statement in Theorem~\ref{thm:HDGauss_studentizedts}.
\end{proof}

\begin{proof}[Proof of Theorem~\ref{thm:mainthBts2}:]
We adapt (mostly notationally) the argument used in the proof of Theorem~\ref{thm:mainthB2}. The present proof combines Theorem~\ref{thm:mainthBts} with Lemma~\ref{lem:auxB} (we only need to consider Case 1 in the proof of Theorem~\ref{thm:mainthBts}, i.e., when the condition in~\eqref{eq:epscond'ts} is satisfied, due to the structure of the bound claimed in Theorem~\ref{thm:mainthBts2}). We assume, without loss of generality, that~$\max(\eps_V, \eps_V': V \in \{X, Y\}) < 1/2$ and~$J(\Delta) \neq \emptyset$. Note that~$J(\Delta) \not \subseteq \hat{\Pi}^{\Delta}_{B,e}$ implies the existence of an index~$\hat{j} = j(\tilde{X}_1,\hdots, \tilde{X}_{n_X}, \tilde{Y}_1,\hdots, \tilde{Y}_{n_Y})\in J(\Delta)$ and a step~$\hat{k}=k(\tilde{X}_1,\hdots, \tilde{X}_{n_X}, \tilde{Y}_1,\hdots, \tilde{Y}_{n_Y})$, say, in the while loop of Algorithm~\ref{algo:bootts2}, at which
\begin{equation*}
	|S^{\Delta}_{W,\hat{j}}|> \hat{c}_{1-\alpha, \hat{\Pi}_{\hat{k}}}(\tilde{\Sigma}^{\Delta}_{0}) \quad \text{ and } \quad J(\Delta)\subseteq \hat{\Pi}_{\hat{k}}.
\end{equation*}
By definition~$\hat{c}_{1-\alpha, \hat{\Pi}_{\hat{k}}}(\tilde{\Sigma}^{\Delta}_{0})$ is the~$m(B, \alpha, d)$th order statistic from the i.i.d.~sample~$$\|z_1^{(\hat{k})}\|_{\infty, \hat{\Pi}_{\hat{k}}}, \hdots, \|z_B^{(\hat{k})}\|_{\infty, \hat{\Pi}_{\hat{k}}},$$ which obviously dominates the~$m(B, \alpha, d)$th order statistic~$o_{J(\Delta)}^{(\hat{k})}$, say, of the i.i.d.~sample~$$\|z_1^{(\hat{k})}\|_{\infty, J(\Delta)}, \hdots, \|z_B^{(\hat{k})}\|_{\infty, J(\Delta)}.$$ We hence obtain
\begin{equation*}
	\max_{j\in J(\Delta)}|S^{\Delta}_{W,j}|
	\geq
	|S^{\Delta}_{W,\hat{j}}|
	> 
	\hat{c}_{1-\alpha, \hat{\Pi}_{\hat{k}}}(\tilde{\Sigma}^{\Delta}_{0})
	\geq 
	o_{J(\Delta)}^{(\hat{k})},
\end{equation*}
and conclude (with the argument used in Case 1 in the proof of Theorem~\ref{thm:mainthBts} applied to the first summand in the first subsequent upper bound) that for a suitable constant, with the required dependence properties,
\begin{equation}
	\begin{aligned}\label{eqn:redstep2}
		\P\left(J(\Delta) \not \subseteq \hat{\Pi}^{\Delta}_{B,e}\right)
		& \leq
		\P \left(\max_{j\in J(\Delta)}|S^{\Delta}_{W,j}|>c_{1-\alpha,  J(\Delta)}(\tilde{\Sigma}^{\Delta}_{0})\right) + \P\left(c_{1-\alpha,  J(\Delta)}(\tilde{\Sigma}^{\Delta}_{0}) > o_{J(\Delta)}^{(\hat{k})}\right) \\
		& \leq \alpha +  C \times \left(\mathfrak{A}^X_{n_X} + \mathfrak{A}^Y_{n_Y}+\mathfrak{B}^{\Delta} + \mathfrak{C}^{\Delta}\right) + \P\left(c_{1-\alpha,  J(\Delta)}(\tilde{\Sigma}^{\Delta}_{0}) > o_{J(\Delta}^{(\hat{k})}\right).
	\end{aligned}
\end{equation}
It remains to verify that
\begin{equation}\label{eqn:lastin}
\P\left(c_{1-\alpha,  J(\Delta)}(\tilde{\Sigma}^{\Delta}_{0}) > o_{J(\Delta)}^{(\hat{k})}\right) \leq B^{-1}.
\end{equation}
There are at most~$d$ steps in Algorithm~\ref{algo:bootts2}, and we can implicitly imagine~$j = 1, \hdots, d$ independent underlying samples of i.i.d.~random vectors~$z_1^{(j)}, \hdots, z_B^{(j)} \sim \mathsf{N}_{d}(0, \tilde{\Sigma}^{\Delta}_{0})$ conditionally on the observations. Let us denote by~$o_{J(\Delta)}^{(k)}$ the $m(B, \alpha, d)$th order statistic from~$$\|z_1^{(k)}\|_{\infty, J(\Delta)}, \hdots, \|z_B^{(k)}\|_{\infty, J(\Delta)}.$$ By the independence and distributional assumptions imposed on the random variables generated in Algorithm~\ref{algo:bootts2}, abbreviating~$\sigma(\tilde{X}_1, \hdots, \tilde{X}_{n_X}, \tilde{Y}_1, \hdots, \tilde{Y}_{n_Y}) =: \mathcal{O}$,
\begin{align*}
	\P\left(c_{1-\alpha,  J({\Delta})}(\tilde{\Sigma}^{\Delta}_{0}) > o_{J(\Delta)}^{(\hat{k})} \mid \mathcal{O} \right) &\leq \P\left(c_{1-\alpha,  J(\Delta)}(\tilde{\Sigma}^{\Delta}_{0}) > \min_{j = 1}^d o_{J(\Delta)}^{(j)} \mid \mathcal{O} \right) \\
	&= 1 - \P\left(c_{1-\alpha,  J(\Delta)}(\tilde{\Sigma}^{\Delta}_{0}) \leq \min_{j = 1}^d o_{J(\Delta)}^{(j)} \mid \mathcal{O} \right) \\
	&= 1 - \P^d\left(c_{1-\alpha,  J(\Delta)}(\tilde{\Sigma}^{\Delta}_{0}) \leq o_{J(\Delta)}^{(1)} \mid \mathcal{O}  \right).
\end{align*}
But Lemma~\ref{lem:auxB} applied (conditionally on~$\mathcal{O}$) and with~$z_i := \|z_i^{(1)}\|_{\infty, J(\Delta)}$ for~$i = 1, \hdots, B$,~$\beta = 1-\alpha$ and~$\gamma = (1-1/B)^{1/d}$, noting that then~$c_{1-\alpha,  J(\Delta)}(\tilde{\Sigma}^{\Delta}_{0}) = Q_F(1-\alpha)$ (in the notation of Lemma~\ref{lem:auxB}), delivers 
\begin{equation*}
	(1-1/B)^{1/d} = \gamma \leq 
	\P\left(z^*_{Q_{B, \beta}(\gamma) + 1} \geq Q_F(\beta) \mid \mathcal{O} \right) =  
	\P\left(o_{J(\Delta)}^{(1)} \geq c_{1-\alpha,  J(\Delta)}(\tilde{\Sigma}^{\Delta}_{0}) \mid \mathcal{O} \right).
\end{equation*}
Hence,
$$
\P\left(c_{1-\alpha,  J({\Delta})}(\tilde{\Sigma}^{\Delta}_{0}) > o_{J(\Delta)}^{(\hat{k})} \mid \mathcal{O} \right) \leq B^{-1},
$$
which implies~\eqref{eqn:lastin}.
\end{proof}

\end{appendix}

\end{document}